\documentclass[11pt,a4paper]{amsart}
\usepackage{amssymb,amsmath}
\usepackage{graphicx}
\usepackage{bm,bbm}
\usepackage{color}
\usepackage{tikz}
\usepackage{ifthen}
\usetikzlibrary{snakes}
\usetikzlibrary{patterns}
\usepackage{pgfplots}
\usepackage{enumitem}
\usepackage{array,blkarray}
\usepackage{bigdelim}
\usepackage{listings} 
\usepackage{verbatim}
\usepackage[hidelinks]{hyperref}
\usepackage{mathtools}  
\newcommand{\R}{\mathbb{R}}
\newcommand{\C}{\mathbb{C}}
\newcommand{\A}{\mathcal{A}}
\newcommand{\D}{\mathcal{D}}
\newcommand{\I}{\mathcal{I}}
\newcommand{\eps}{\varepsilon}
\newcommand{\ci}{\mathrm{i}}
\newcommand{\nablaR}{\nabla_{\hspace{-2pt}\mbox{\rm \tiny R}}}
\newcommand{\WR}{W_{\hspace{-2pt}\mbox{\rm \tiny R}}}

\def\dx{\,\text{d}x}

\definecolor{olive}{RGB}{114,175,30} 
\definecolor{orange}{RGB}{225,92,22} 
\definecolor{dark-green}{rgb}{0.0,0.4,0.0}
\definecolor{darkBlue}{rgb}{0.0,0.0,0.6}
\definecolor{darkRed}{rgb}{0.8,0.1,0.0}

\newcommand{\Hscapro}[2]{ \Re(  #1 , #2 )_{L^2(\D)} }
\newtheorem{theorem}{Theorem}[section]
\newtheorem{lemma}[theorem]{Lemma}

\newtheorem{conclusion}[theorem]{Conclusion}
\newtheorem{example}[theorem]{Example}

\newtheorem{proposition}[theorem]{Proposition}
\begin{document}
\title[$J$-Method for Gross-Pitaevskii Eigenvalue Problem]{The $J$-Method for the\\ Gross-Pitaevskii Eigenvalue Problem}
\author[]{R.~Altmann$^*$, P.~Henning$^{\dagger}$, D.~Peterseim$^*$}
\address{${}^{*}$ Department of Mathematics, University of Augsburg, Universit\"atsstr.~14, 86159 Augsburg, Germany}
\address{${}^{\dagger}$ Department of Mathematics, Ruhr-University Bochum, DE-44801 Bochum, Germany and Department of Mathematics, KTH Royal Institute of Technology, SE-100 44 Stockholm, Sweden}
\email{\{robert.altmann, daniel.peterseim\}@math.uni-augsburg.de, patrick.henning@rub.de}
\thanks{P.~Henning acknowledges funding by the Swedish Research Council (grant 2016-03339). The work of D.~Peterseim is part of a project that has received funding from the European Research Council (ERC) under the European Union's Horizon 2020 research and innovation programme (Grant agreement No.~865751).}
\date{\today}
\maketitle
%
\begin{abstract}
This paper studies the $J$-method of [\emph{E.~Jarlebring, S.~Kvaal, W.~Michiels. SIAM J.~Sci.~Comput.~36-4:A1978-A2001, 2014}] for nonlinear eigenvector problems in a general Hilbert space framework. This is the basis for variational discretization techniques and a mesh-independent numerical analysis. 
A simple modification of the method mimics an energy-decreasing discrete gradient flow. In the case of the Gross-Pitaevskii eigenvalue problem, we prove global convergence towards an eigenfunction for a damped version of the $J$-method. More importantly, when the iterations are sufficiently close to an eigenfunction, the damping can be switched off and we recover a local linear convergence rate previously known from the discrete setting. This quantitative convergence analysis is closely connected to the~$J$-method's unique feature of sensitivity with respect to spectral shifts. Contrary to classical gradient flows, this allows both the selective approximation of excited states as well as the amplification of convergence beyond linear rates in the spirit of the Rayleigh quotient iteration for linear eigenvalue problems. These advantageous convergence properties are demonstrated in a series of numerical experiments involving exponentially localized states under disorder potentials and vortex lattices in rotating traps.
\end{abstract}
%
\section{Introduction}
This paper studies eigenvalues of nonlinear differential operators $\A$  on a {\it real} Hilbert space $V$ of the form
$$
\A (v) = A(v,v),
$$
where $A\colon V \times V \rightarrow V^*$ is a continuous, bounded mapping that is invariant under scaling of the first argument and real-linear in the second argument. A well-known example that can be cast in such a format is the Gross-Pitaevskii eigenvalue problem (GPEVP) in a bounded domain $\D\subset \R^d$ (with $d=2,3$). In the strong form the GPEVP reads
\begin{align}
\label{eqn:GPEVP:strong}
\left[-\Delta + W  - \Omega \mathcal{L}_z + \kappa\, |u^*|^2\right] u^* 
= \lambda^* u^*.
\end{align}
Here, $W \in L^{\infty}(\D,\R)$ denotes a non-negative and space-dependent potential, $\Omega \in \R$ is the angular velocity, $\mathcal{L}_z = - \ci \left( x \partial_y - y \partial_x \right)$ is the $z$-component of the angular momentum, and $\kappa\ge0$ regulates the nonlinearity of the problem. 
This problem is related, e.g., to the modeling of Bose-Einstein condensates~\cite{Bos24,Ein24,DGP99,PiS03}. 
In this context, a solution~$u^*$ represents a stationary quantum state of the condensate, $|u^{*}|^2$ is the corresponding density, and $\lambda^*$ the so-called chemical potential. 

The numerical solution of the GPEVP~\eqref{eqn:GPEVP:strong} has been studied extensively in recent years. Popular methods for solving the nonlinear eigenvalue problem are for example the {\it Self Consistent Field Iteration} (SCF), cf.~\cite{Can00,CaL00,CaL00B,DiC07,UJR18}, which requires to solve a linearized eigenvalue problem in each iteration. Algorithms that belong to the SCF class are for instance the {\it Roothaan algorithm}~\cite{Roothaan51} or the {\it optimal damping algorithm} proposed in~\cite{CaL00}. Another important class of iterative methods is based on gradient flows for the energy functional associated with \eqref{eqn:GPEVP:strong}, where we mention the {\it Discrete Normalized Gradient Flow} (DNGF), cf.~\cite{BaD04,BCL06,BaS08,Bao-et-al-2005}, which is based on an implicit Euler discretization of the $L^2$-gradient flow. Improvements of the approach by using conjugated gradients were proposed in \cite{ALT17}. Here we also mention the {\it Projected Sobolev Gradient Flows} (PSGFs), cf.~\cite{DaK10,GaP01,HSW19,HenP18ppt,KaE10,RSB14,RSS09,Zhang2019}, which form a subclass of the gradient flow methods. Sobolev gradients are the Riesz representants of the Fr\'echet derivative of the energy functional in a suitable Hilbert space. With this, PSGFs are based on computing such Sobolev gradients and then projecting them onto the tangent space of the normalization constraint. An improvement of PSGF by using Riemannian conjugate gradients was suggested in~\cite{DaP17}. Another strategy to solve the GPEVP involves a direct minimization of the energy functional, cf.~\cite{BaT03,COR09}, which means that \eqref{eqn:GPEVP:strong} is written as a nonlinear saddle point problem which is then solved by a Newton-type method. Global convergence results for solving the GPEVP are very rare in the literature. To the best of our knowledge, global convergence had been so far only established for a damped PSGF suggested in \cite{HenP18ppt}, where corresponding analytical results can be found in \cite{HenP18ppt,Zhang2019}.

The SCF and PSGF approaches are typically based on the simple linearization of the partial differential operator by replacing the density $|u^{*}|^2$ in \eqref{eqn:GPEVP:strong} by the density of some given approximation of $u^*$. 
In the finite-dimensional case (after spatial discretization) these approaches can be interpreted as generalizations of the inverse iteration (inverse power method). Linear convergence is observed empirically which is in agreement with the convergence theory for linear matrix eigenvalue problems. However, in the presence of clustered eigenvalues, which are typically connected to interesting physical phenomena, linear convergence can be very slow and shifting (as in the shifted inverse iteration or the Rayleigh quotient iteration) is a well-established technique for the acceleration of matrix eigenvalue solvers. The problem is that the aforementioned schemes seem to be unable to achieve any speed-up by a sophisticated shift of the nonlinear operator. In~\cite{JarKM14}, Jarlebring, Kvaal, and Michiels observed that this insensitivity towards spectral shifts is the result of an unsuitable choice of linearization. The authors propose a natural but conceptually different linearization using the Jacobian of the nonlinear operator~$\A$. The resulting approach does not belong to any of the classes above and we will refer to it as the {\em $J$-method}. 

This paper generalizes the $J$-method and its numerical analysis to an abstract Hilbert space setting. Hence, the resulting iteration scheme is based on a variational formulation which is mesh-independent (cf.~\cite{AltF20} for a similar approach in electromagnetism). This then allows the application of any spatial discretization, including finite elements and spectral methods. 
As discovered in ~\cite{JarKM14}, the correct choice of the linearization in the $J$-method does not only lead to a competitive nonlinear eigenvalue solver, it also allows major theoretical progress such as a proof of local linear convergence based on a version of Ostrowski's theorem (see~Section~\ref{sect:local} for the corresponding result the Hilbert space setting). The obtained linear convergence rate, which is essentially~$|\lambda^* + \sigma|/|\mu + \sigma|$, depends on the spectral gap of the $\sigma$-shifted Jacobian, where~$\mu + \sigma$ is the second-smallest eigenvalue of the $\sigma$-shifted Jacobian around~$u^*$. Our results generalize the findings of~\cite{JarKM14} from the matrix case to the abstract setting. In the case of the GPEVP, which is discussed in~Section~\ref{sect:localGPE}, this marks the first quantified convergence result. Moreover, an adaptive choice of the shifts in the spirit of the Rayleigh quotient iteration amplifies convergence beyond the linear rate in representative numerical experiments, see Section~\ref{sect:numerics}. At the same time, the sensitivity to spectral shifts facilitates the computation of excited states. 

As usual for nonlinear problems, the quantitative convergence results are of local nature in the sense that they require a sufficiently accurate initial approximation. In practical computations this may be a rather challenging task. A simple damping strategy inspired by an energy-dissipative gradient flow \cite{HenP18ppt} resolves this problem. In the case of the GPEVP we prove global convergence of the method towards an eigenfunction with a guaranteed decrease of energy, cf.~Section~\ref{sect:GPE}. 

Altogether, the combination of the $J$-method with damping for globalization and shifting for acceleration provides a powerful methodology for the simulation and analysis of nonlinear PDE eigenvector problems such as the GPEVP. 
%
%
\section{Definition of the $J$-method in Hilbert spaces}
%
\subsection{Nonlinear eigenvector problem}
We consider a nonlinear differential operator 
$$
\A\colon V \rightarrow V^*
$$
that maps a {\it real} Hilbert space $V$ into its dual space $V^*$. In particular, $V$ is equipped with an $\R$-inner product and every $F\in V^*$ satisfies $\langle F, v\rangle_{V^*,V} \in \R$ for all~$v\in V$. 
Note that $V$ is still allowed to contain complex-valued functions. We assume that $\A$ has the form
$$
\A (v) = A(v,v),
$$
where $A\colon V \times V \rightarrow V^*$ is continuous, bounded, and (real-)linear in the second argument. Recall that $A(v,\,\cdot\,)$ being real-linear (cf.~\cite[Def.~3.50]{BCS17}) means that for all $v,w_1,w_2\in V$ and $\alpha \in \R$ we have
$$
A(v,w_1+w_2) = A(v,w_1)+A(v,w_2) \qquad \mbox{and} \qquad A(v,\alpha w_1) = \alpha A(v, w_1).
$$
Furthermore, we assume that $A(\,\cdot\, , \cdot\,)$ is sufficiently smooth on $V\setminus\{0\} \times V$ (in particular real-Fr\'echet differentiable in both arguments) and that it is real-scaling invariant in the first argument, i.e.,
$$
A(u,v) = A(\alpha u , v) 
$$
for all $u,v \in V$ and all $\alpha \in \R \setminus \{ 0\}$. 

For the weak formulation of the eigenvalue problem with a nonlinearity in the eigenfunction, 
we introduce another real Hilbert space $H$ (the pivot space) with $\R$-inner product
$$
(\,\cdot\,,\,\cdot\,)_H\colon H \times H \rightarrow \R.
$$
such that $V, H, V^*$ forms a Gelfand triple \cite[Ch.~23.4]{Zei90a}. The corresponding $H$-norm $\|\cdot\| := \|\cdot\|_H$ will be used for the normalization condition of the eigenfunction. The corresponding embedding $\I\colon V\to V^*$ is defined via 
\[
\langle \mathcal{I} u,\, \cdot\, \rangle_{V^*,V} = (u ,\, \cdot\, )_{H}
\]
for $u\in V$. For brevity, we shall write $\langle\, \cdot\,,\, \cdot \rangle$ for the canonical duality pairing in $V$.

The goal of this paper is to solve the corresponding PDE eigenvalue problem: find an eigenfunction $u^*\in V$ with $\|u^* \|=1$ and an eigenvalue $\lambda^{\ast}\in \R$ such that $\A(u^*) = \lambda^* \I u^*$. This is equivalent to the variational formulation 
\begin{align}
\label{eqn:generalEVP}
\langle \A(u^*), v \rangle 
= \langle \lambda^* \I u^*, v \rangle 
= \lambda^* (u^*, v)_H \quad\text{for all test functions } v\in V. 
\end{align}

Throughout the paper, we denote by $\overline{\alpha}$ the complex conjugate of a complex number $\alpha\in \C$, by $\Re(v)$ the real part of $v$, and by $\Im(v)$ the imaginary part.
%
\begin{example}
\label{exp:GPEVP}
We are particularly interested in the GPEVP~\eqref{eqn:GPEVP:strong}. Here we consider $H=L^2(\D):=L^2(\D,\C)$ as a {\rm real} Hilbert space equipped with the inner product 
$$
(v,w)_{H}
:= \Re \left(\int_{\D}  v \hspace{2pt} \overline{w} \dx \right)
$$ 
and analogously $V=H^1_0(\D):=H^1_0(\D,\C)$ as a real Hilbert space equipped with
$$
(v,w)_{V} 
:= \Re \left(\int_{\D} \nabla v \cdot \overline{\nabla w} \dx \right).
$$ 
The (real) dual space is denoted by $V^*= H^{-1}(\D):=H^{-1}(\D,\C)$. The corresponding nonlinear operator reads
\begin{align}
\label{eqn:GPEVP:weak}
\langle A(u,v) , w \rangle 
:= \Re \left( \int_{\D} \nablaR v \cdot \overline{\nablaR w} + \WR \, v\, \overline{w} \dx  
+ \frac{\kappa}{\| u \|^2} \int_{\D} |u|^2 v\, \overline{w}\dx \right).
\end{align}
Here, we have $\WR(x)= W(x) - \tfrac{1}{4}\, \Omega^2\, |x|^2$ and the rotational gradient $\nablaR$ is given by
\begin{align*}
\nablaR v := \nabla v + \ci \frac{\Omega}{2} R^{\top} v
\end{align*}
for the divergence-free vector field $R(x,y,z):=(y,-x,0)$ if~$d=3$ and $R(x,y):=(y,-x)$~if $d=2$. Note that the nonlinear term is multiplied by $\| u\|^{-2}$ in order to achieve the assumed scaling invariance of $A$ in the first component. 
\end{example}
%
\subsection{Linearization and the $J$-operator}
Recall that the real-Fr\'echet derivative of some $F\colon X \rightarrow Y$ in a point $u \in X$ is a bounded and $\R$-linear operator $F^{\prime}(u)\colon X \rightarrow Y $ with the property 
\begin{align*}
\lim_{h\in X\setminus \{0\},\, h \rightarrow 0} \frac{ \| F(u+h) - F(u) - F^{\prime}(u)h \|_Y }{ \| h \|_{X} } = 0. 
\end{align*}
We write $F^{\prime}(u;h):=F^{\prime}(u)h$ for the Fr\'echet derivative in $u$ in direction of~$h$.
In the following, we speak about the Fr\'echet derivatives of operators, where we always mean the real-Fr\'echet derivative as stated above. 
Moreover, we neglect writing the supplement $h\in X\setminus \{0\}$ beneath the limit.
For the operator $A$, we denote the partial Fr\'echet derivative with respect to the (nonlinear) first component by $\partial_1 A\colon V^3 \rightarrow V^*$. Because of the scaling invariance of $A$, i.e., $A(u,\cdot\,)=A(\alpha u,\cdot\,)$, we have~$\partial_1  A(u,\cdot\,;u)=0$,
which can be seen by
$$
\lim_{t\rightarrow 0} \frac{\| A(u+tu,\cdot\,) - A(u,\cdot\,) - 0\|_{\mathcal{L}(V,V^{\ast}) } }{t\, \| u \|_{V}} 
= \lim_{t\rightarrow 0} \frac{\|A(u,\cdot\,) - A(u,\cdot\,)\|_{\mathcal{L}(V,V^{\ast}) } }{t\, \|u\|_V} 
= 0.
$$
The Fr\'echet derivative of $\A$ in $u\in V$ and in direction $w \in V$ can be expressed as 
\begin{align*}
\A^{\prime}(u; w) 
&=  \lim_{t\rightarrow 0} \frac{A(u+tw,u+tw) - A(u,u)}{t}\\
&=  \lim_{t\rightarrow 0} \frac{A(u+tw,u) - A(u,u)}{t} + \lim_{t\rightarrow 0} \frac{A(u+tw,u+tw) - A(u+tw,u)}{t}  \\
&=  \partial_1  A(u, u; w)  + A(u,w). 
\end{align*}
Due to $\partial_1  A(u,\cdot\,;u) = 0$, we conclude
\[
\A^{\prime}(u;u) 
= \partial_1  A(u,u;u)  + A(u,u) 
= A(u,u)
= \A(u).
\]
%
For an element $u\in V$ we now define the operator $J(u)\colon V \rightarrow V^*$ by
\begin{align}
\label{abstract-definiton-J}
J(u) := \A^{\prime}(u;\, \cdot\,). 
\end{align}
Hence, using $\langle J(u^{\ast} ) u^{\ast} , v \rangle  = \langle \A^{\prime}(u^{\ast},u^{\ast}) , v \rangle = \langle \A(u^{\ast}), v \rangle$,
we observe that the eigenvalue problem~\eqref{eqn:generalEVP} can be rewritten as: find $u^{\ast} \in V$ with $\| u^{\ast} \|=1$ and $\lambda^{\ast}\in \R$ such that
\begin{align}
\label{eqn:generalEVP:J}
\langle J(u^{\ast} ) u^{\ast} , v \rangle 
= \lambda^{\ast} (u^{\ast},v)_{H}
\qquad \mbox{for all } v \in V.
\end{align}
%
%
\subsection{The shifted $J$-method}\label{ss:J-shifted}
The shifted $J$-method is defined by applying the inverse power iteration to the reformulated eigenvalue problem \eqref{eqn:generalEVP:J} based on the linearization~$J$ and some spectral shift~$\sigma$. For that, we consider a function $v\in V$ and~$\sigma\in\R$ such that $J(v) +\sigma \I\colon V\to V^*$ is invertible. 
In practice, one may either choose~$\sigma$ large such that~$J(v) +\sigma \I$ is coercive (see Section~\ref{sect:GPE:coercive} for the example of the GPEVP) or close to the negative of the target eigenvalue (as it is done in the numerical experiments of Section~\ref{sect:numerics}).   

For $F\in V^*$ we define $u = (J(v) +\sigma \I)^{-1}F \in V$ as the unique solution of the variational problem 
$$
\langle (J(v) +\sigma \mathcal{I}) u , w \rangle 
= \langle F , w \rangle 
$$
for all $w \in V$. Considering a fixed shift $\sigma$ such that~$-\sigma$ is no eigenvalue of $J(v)$, we define $\psi\colon V\to V$ by 
$$
\psi(v) 
:= (J(v) +\sigma \mathcal{I} )^{-1} \I v. 
$$
Including an additional normalization step finally leads to the operator $\phi\colon V\to V$, 
$$
\phi(v) 
:= \frac{\psi(v)}{\| \psi(v) \|}. 
$$
Now consider a normalized eigenpair $(u^*, \lambda^*)$ of the eigenvalue problem~\eqref{eqn:generalEVP}. Then, $\psi(u^*)$ satisfies due to $J(u^*)u^* = \A(u^*) = \lambda^* \I u^*$, 
\begin{align*}
(\lambda^*+\sigma) \psi(u^*)
&= (J(u^*) +\sigma \I)^{-1} J(u^*)\, u^* + \sigma (J(u^*) +\sigma \I)^{-1} \I u^* \\
&= (J(u^*) +\sigma \I)^{-1} (J(u^*) + \sigma\I)\, u^*
= u^*. 
\end{align*}
As a consequence, we observe that $u^{\ast}$ is a fixed point of $\phi$, i.e., $u^{\ast} = \phi(u^{\ast})$. This motivates the following iteration scheme: given $u^0\in V$ with $\|u^0\|=1$ and a shift $\sigma$ such that $J(u^0) +\sigma \I$ is invertible, compute iterates $u^{n+1}$ for $n\geq0$ by
\begin{align}
\label{eqn:Jmethod}
u^{n+1} 
= \phi(u^n)
= \frac{\psi(u^n)}{\| \psi(u^n) \|}
= \alpha_n\, (J(u^n) +\sigma \I )^{-1} \I u^n 
\end{align}
with the normalization factor $\alpha_n = 1 / \| (J(u^n) +\sigma \I )^{-1} \I u^n\|$. For the matrix case, the convergence of this scheme was analyzed in~\cite{JarKM14}. We emphasize that this scheme is well-defined if the shift~$\sigma$ guarantees the invertibility of $J(u^n) +\sigma \I$. This property has to be checked within the specific application and will be proved for the GPEVP in Section~\ref{sect:GPE:coercive}. 
%
\subsection{The damped $J$-method}\label{ss:J-damped}
In many applications, eigenvalue problems (with a nonlinearity in the eigenfunction) can be equivalently formulated as a constrained energy minimization problem. In this case, the radius of convergence can be considerably enhanced by introducing a variable damping parameter $\tau_n$.  This damping parameter (or step size) is adaptively selected such that the energy is optimally minimized in each iteration. In \cite{HenP18ppt}, a geometric justification of this approach was given, by interpreting it as the discretization of a certain projected Sobolev gradient flow, where the inner product (with which respect the Sobolev gradient is computed) is based on a repeated linearization of the differential operator. 
Even though $\langle (J(v) +\sigma \I)\, \cdot\, , \cdot\, \rangle $ typically does not define an inner product (and hence does not allow a geometric interpretation), it is still possible to generalize the $J$-method defined in~\eqref{eqn:Jmethod} in the spirit of the results in~\cite{HenP18ppt}. In Section~\ref{sect:GPE} below, we shall give a rigorous justification of this approach by proving global convergence of the damped $J$-method in the context of the GPEVP.

The damping strategy aims to find a (optimized) linear combination of $u^n$ and $(J(u^n) +\sigma \I)^{-1}\I u^n$. For brevity we introduce $J_\sigma(u) := J(u)+\sigma \I$ and assume for a moment that the shift $\sigma$ is chosen such that $J_\sigma(u^n)$ remains coercive.  

One step of the damped $J$-method then reads as follows: given $u^n\in V$ with $\|u^n\|=1$ and a step size $\tau_n>0$ compute 
\begin{align}
\label{eqn:JmethodDamped}
\tilde{u}^{n+1} 
= (1-\tau_n)u^n + \tau_n\, \gamma_n\, J_\sigma(u^n)^{-1} \I u^n, \qquad 
u^{n+1} = \frac{\tilde{u}^{n+1}}{\| \tilde{u}^{n+1} \|}
\end{align} 
with $\gamma_n^{-1} := (J_\sigma(u^n)^{-1} \I u^n , u^n)_H$. This choice provides the important $H$-orthogonality  
\begin{align}
\label{eqn:L2-orthogonality} 
( \tilde{u}^{n+1} - u^n, u^n)_H = 0.
\end{align} 
Due to the normalization in~\eqref{eqn:JmethodDamped}, we recover the original $J$-method~\eqref{eqn:Jmethod} for $\tau_n=1$. 
Further, the $H$-orthogonality~\eqref{eqn:L2-orthogonality} implies that the mass (i.e., the~$H$-norm) of $\tilde{u}^{n+1}$ is greater or equal to~$1$ (even strictly larger unless ${u}^{n+1}=u^{n}$). This can be seen from 
\begin{align}
\label{intermediate-mass-growth}
1 = \| u^{n} \|^2 = ( u^{n} , \tilde{u}^{n+1} )_H \le \| \tilde{u}^{n+1}\|.
\end{align}
For the example of the GPEVP we will show in Section~\ref{sect:GPE} that the iteration~\eqref{eqn:JmethodDamped} is well-defined and that it converges to an eigenstate if the step size $\tau_n$ is sufficiently small. 

To summarize, we have introduced a damped version of the $J$-method~\eqref{eqn:JmethodDamped} (with typical choice $\sigma=0$ or $\sigma$ large) which intends to ensure global convergence and a shifted version~\eqref{eqn:Jmethod} to accelerate the convergence (with typical choice $\sigma$ close to $-\lambda^{\ast}$). 
Note, however, that we do not intend to use damping and shifting simultaneously as this seems hard to control numerically. 
Instead, we propose to apply damping for globalization of convergence and then switch to shifting if a certain accuracy is obtained. This practice then provides a powerful methodology as illustrated in Section~\ref{sect:numerics}.
%
%
\section{Abstract local convergence of the shifted $J$-method}\label{sect:local}
The sensitivity of the~$J$-method to spectral shifts facilitates the numerical approximation of excited states. To see this, we transfer the local convergence result presented in~\cite{JarKM14} to the Hilbert space setting. This shows that the shifted $J$-method converges to an eigenfunction $u^{\ast}$ of $\mathcal{A}$ if the starting function and the shift are sufficiently close to the eigenpair $(u^{\ast},\lambda^{\ast}$) of interest.
In this section, we consider the undamped version of the $J$-method, i.e., we consider the iteration~$u^{n+1} = \phi(u^n)$, including an arbitrary shift $\sigma \in \R$ and assuming that $J_\sigma(u^n)$ is invertible.
In order to prove local convergence with a linear rate that depends on the shift, we make use of the following well-known lemma that is a version of the Ostrowski theorem in Banach spaces.
\begin{proposition}
\label{prop-ostrowski}
Let $\phi\colon V \rightarrow V$ be a mapping on the Banach space $V$ that is Fr\'echet-differentiable at a point $u^{\ast} \in V$ such that the Fr\'echet-derivative $\phi^{\prime}(u^{\ast})\colon V \rightarrow V$ is a bounded linear operator with spectral radius 
$
\rho^{\ast}:= \rho( \hspace{2pt} \phi^{\prime}(u^{\ast}) \hspace{2pt}) < 1.
$
Then there is an open neighborhood $S$ of $u^{\ast}$, such that for all starting values $u^{0} \in S$ we have that the fixed point iterations
$$
u^{n+1} := \phi( u^n )
$$ 
converge strongly in $V$ to $u^{\ast}$, i.e., $\| u^n - u^{\ast} \|_V \rightarrow 0$ for $n\rightarrow \infty$. Furthermore, for every $\eps>0$ there exist a neighborhood $S_{\eps}$ of $u^{\ast}$ and a (finite) constant $C_\eps>0$ such that
$$
\|  u^n - u^{\ast} \|_{V} \le C_\eps \hspace{2pt} |\rho^{\ast} + \eps|^n \hspace{2pt} \| u^0 - u^{\ast} \|_{V} \qquad \mbox{for all } u^0 \in S_{\eps}
$$
and $n\ge 1$.
Hence, $\rho^{\ast}$ defines the asymptotic linear convergence rate of the fixed point iteration. 
\end{proposition}
\begin{proof}
The proof is based on the observation that under the assumptions of the lemma and for each $\eps>0$, there exists an induced operator norm $\| \cdot \|_{\eps}$ such that $\| \phi^{\prime}(u^{\ast}) \|_{\eps} \le \rho^{\ast} + \eps$, cf.~\cite[Lem.~4.3.7]{AMR88}. With this result at hand, we can exploit the norm equivalence together with the definition of Fr\'echet derivatives to conclude the desired local convergence rate. This simple argument is elaborated e.g.~in~\cite{Shi81,Fin96} and generalizes the previous findings obtained in~\cite{Kit66}. 
\end{proof}
In the finite dimensional case and under some additional assumptions on the structure of $\phi^{\prime}(u^{\ast})$ it can be shown that $\|  u^n - u^{\ast} \|_{V} \le C \hspace{2pt} |\rho^{\ast}|^n \| u^0 - u^{\ast} \|_{V}$, cf.~\cite[Th.~10.1.3 and 10.1.4; NR 10.1-5]{OrR70}.

In the spirit of Proposition~\ref{prop-ostrowski} we need to study the spectrum of $\phi^{\prime}(u^{\ast})$ to conclude local convergence.
\subsection{Derivative of $\phi$}
Since we interpret $\psi$ and $\phi$ as operators from $V$ to $V$, we have likewise for the first Fr\'echet derivative in $v\in V \setminus \{0\}$,
$
\psi'(v\,;\, \cdot\,),\ \phi^{\prime}(v\,;\, \cdot\,) 
\colon V \to V. 
$
In order to apply Proposition~\ref{prop-ostrowski}, we need to compute the Fr\'echet derivative of $\phi$ in $u^{\ast}$. As a first step, we compute the derivative of the mapping $v \mapsto \| \psi(v) \|$, leading to  
\begin{align*}
\langle \partial_v \| \psi(v) \| , w \rangle 
= \frac{(\psi^{\prime}(v;w) , \psi(v))_H}{\| \psi(v) \|}
= (\psi^{\prime}(v;w) , \phi(v))_H
\end{align*}
or, more compactly, $\partial_v \| \psi(v) \| = (\psi^{\prime}(v), \phi(v))_H$. In order to compute $\phi^{\prime}(v)$, we write $\psi^{\prime}(v)$ in the form 
\begin{align*}
\psi^{\prime}(v) 
= \partial_v \left( \phi(v) \| \psi(v) \| \right)
= \| \psi(v)\|\, \phi^{\prime}(v) + (\psi^{\prime}(v) , \phi(v))_H \, \phi(v)
\end{align*}
and conclude that
\begin{align*}
\phi^{\prime}(v) 
= \frac{\psi^{\prime}(v)}{ \| \psi(v) \| } - \frac{1}{ \| \psi(v) \| }
(\psi^{\prime}(v) , \phi(v))_H\, \phi(v).
\end{align*}
For $v=u^{\ast}$ we can exploit $\phi(u^{\ast}) = u^{\ast}$, $\| u^{\ast} \|=1$, and $u^* = (\lambda^*+\sigma)\, \psi(u^{\ast})$. This implies $\| \psi(u^{\ast}) \| = |\lambda^{\ast} + \sigma|^{-1}$ and thus, 
\begin{align*}
\phi^{\prime}(u^{\ast}) 
= |\lambda^{\ast} + \sigma| \left( \psi^{\prime}(u^{\ast}) - (\psi^{\prime}(u^{\ast}) , u^{\ast})_H\hspace{2pt} u^{\ast} \right).
\end{align*}
To compute $\psi^{\prime}(u^{\ast})$, in turn, we use the identity 
\begin{align*}
\I 
= \partial_v (\mathcal{I} v)
= \partial_v \left( (J(v) +\sigma \mathcal{I}) \psi(v) \right)
= J^{\prime}(v \hspace{2pt}; \cdot )\, \psi(v) + ( J(v) + \sigma \I)\, \psi^{\prime}(v).
\end{align*}
Note that $J^{\prime}(v)\colon V\times V \rightarrow V^*$ denotes here the Fr\'echet derivative of $J$ in a fixed element $v\in V \setminus \{ 0 \}$. We conclude
\begin{align*}
\psi^{\prime}(u^{\ast}) 
= J_\sigma(u^{\ast})^{-1} \mathcal{I} - J_\sigma(u^{\ast})^{-1}  J^{\prime}(u^{\ast} ; \cdot)\, \psi(u^{\ast}) .
\end{align*}
Next, we want to verify that $J^{\prime}(u^{\ast} ;\, \cdot\,) \psi(u^{\ast})=0$. For that, we consider the definition of $J^{\prime}$ which yields 
\begin{align*}
J^{\prime}(u^{\ast}; w)\, u^{\ast}
&= \lim_{t\rightarrow 0} \frac{J(u^{\ast} + tw)\, u^{\ast} - J(u^{\ast})\, u^{\ast} }{t} \\
&= \lim_{t\rightarrow 0} \frac{J(u^{\ast} + tw)(u^{\ast} + tw) - J(u^{\ast} + tw)(tw) - J(u^{\ast} ) u^{\ast} }{t} \\
&= \lim_{t\rightarrow 0} \frac{\A(u^{\ast} + tw) - \A(u^{\ast}) - J(u^{\ast} + tw)(tw)}{t} \\
&= \lim_{t\rightarrow 0} \left( \A^{\prime}(u^{\ast}) w - J(u^{\ast} + tw)w \right)
\overset{\eqref{abstract-definiton-J}}{=} 0.
\end{align*}
Here we also used that in a neighborhood of $u^*$ (which excludes the zero element due to $\| u^{\ast} \|=1$) the operator~$J$ is continuous. In summary, we proved~$J^{\prime}(u^{\ast};w)\, u^{\ast} = 0$ for all $w\in V$. Since~$u^*$ equals $\psi(u^{\ast})$ up to a multiplicative constant, we conclude that $J^{\prime}(u^{\ast}; \cdot)\, \psi(u^{\ast})=0$ and hence, $\psi^{\prime}(u^{\ast}) = J_\sigma(u^{\ast})^{-1} \I$. Plugging this into previous estimates, we directly obtain 
\begin{align}
\label{eqn:phiPrime}
\phi^{\prime}(u^{\ast}) 
= |\lambda^{\ast} + \sigma| \left( J_\sigma(u^{\ast})^{-1} \I - (J_\sigma(u^{\ast})^{-1} \I, u^{\ast})_H\hspace{2pt} u^{\ast} \right).
\end{align}
%
\subsection{Local convergence results}
Since the $J$-method is of the from $u^{n+1} = \phi(u^n)$, we recall from Proposition~\ref{prop-ostrowski} that the local convergence rate is strongly connected to the  spectrum of~$\phi^{\prime}(u^{\ast})$. Besides the modified expression in \eqref{eqn:phiPrime}, the analysis of this spectrum also requires the following orthogonality result.
\begin{lemma}
\label{lem:leftrigthOrth}	
Let $(u,\lambda) \in V\times \R$ be an eigenpair~of $J(u^*)$. Further, let $(w,\mu) \in V\times \R$ be an adjoint-eigenpair of $J(u^*)$ with $\mu\neq \lambda$, i.e., 
\[
\langle J(u^*)\, v, w\rangle
= \mu\, (v, w)_H
\]
for all $v\in V$. Then, it holds that $(u,w)_H=0$.
\end{lemma}
\begin{proof}
By definition we know that~$\langle J(u^*)\, u, v\rangle = \lambda\, (u,v)_H$ for all $v\in V$. Thus, with the test function $v=w$ we have 
\begin{align*}
\lambda \, (u, w)_H
=  \langle J(u^*)\, u, w\rangle
= \mu \, (u, w)_H. 
\end{align*}
The assumption ${\mu}\neq {\lambda}$ then directly implies the $H$-orthogonality of primal and adjoint eigenfunctions.
\end{proof}
With this, we are prepared to study the spectrum of the linear operator~$\phi^{\prime}(u^*)\colon V\to V$ to make conclusions on the convergence rate using Proposition~\ref{prop-ostrowski}. The corresponding eigenvalue problem reads: 
find $z \in V \setminus \{0 \}$ and $\mu \in \R$ so that
$$
\langle \phi^{\prime}(u^{\ast}) z , v \rangle  = \mu\, (z, v)_H
$$
for all~$v \in V$. 
The subsequent lemma relates the spectra of~$\phi^{\prime}(u^*)$ and ~$J(u^{\ast})$. Recall that~$J(u^{\ast})$ is a linear, bounded operator over~$\R$. Thus, its (real) spectrum coincides with the (real) spectrum of the corresponding adjoint operator~\cite{HuR11}.
\begin{lemma}
\label{lem:local}
Assume that the (real) eigenvalues of~$J(u^{\ast})$ are countable and that there exists a basis of corresponding adjoint-eigenfunctions of $V$. Let the eigenvalue of interest~$\lambda^*$ be simple and $\sigma\in \R$ a shift such that $J_\sigma(u^{\ast})$ is invertible (note that in particular we have $\sigma\neq -\lambda^*$). 
Then, the spectra of~$J(u^{\ast})$ and~$\phi^{\prime}(u^{\ast})$ are connected in the following sense: 
$\mu \neq \lambda^*$ 
is an eigenvalue of~$J(u^{\ast})$ if and only 
if~$|\lambda^*+\sigma| / (\mu+\sigma)$ 
is an eigenvalue of~$\phi^{\prime}(u^{\ast})$ and the eigenvalue~$\lambda^*$ of $J(u^{\ast})$ corresponds to the eigenvalue $0$ of~$\phi^{\prime}(u^{\ast})$. 
\end{lemma}	
\begin{proof}
Due to the assumption on the shift, $J_\sigma(u^{\ast})$ is invertible and~$J_\sigma(u^{\ast})^{-1}\I$ has the eigenvalues $(\lambda_1+\sigma)^{-1}, (\lambda_2+\sigma)^{-1}, \dots$, including $(\lambda^*+\sigma)^{-1}$ with eigenfunction~$u^*$. Note that the eigenvalues are shifted but the primal- and adjoint-eigenfunctions are the same as for~$J(u^{\ast})$.
We first consider the eigenvalue~$(\lambda^*+\sigma)^{-1}$ of 
$J_\sigma(u^{\ast})^{-1}\I$ 
with eigenfunction~$u^*$. Here we note that 
\[
\phi^{\prime}(u^*)\, u^*
\overset{\eqref{eqn:phiPrime}}{=} |\lambda^* + \sigma|
\left( J_\sigma(u^*)^{-1} \I u^* - (J_\sigma(u^*)^{-1} \I u^*, u^*)_H\, u^* \right)
= 0,
\]
i.e., $0$ is the corresponding eigenvalue of~$\phi^{\prime}(u^*)$. Now let $(\mu+\sigma)^{-1}$ be an eigenvalue of 
$J_\sigma(u^{\ast})^{-1}\I$ with $\mu\neq \lambda^*$. Since~$J_\sigma(u^{\ast})^{-1}\I$
is a linear, bounded operator over~$\R$, its (real) spectrum coincides with the (real) spectrum of the corresponding adjoint operator (cf.~\cite{HuR11}) such that there exists a corresponding adjoint-eigenfunction~$w$. By definition this means that 
\[
\big(J_\sigma(u^*)^{-1}\I v, w \big)_H = \frac{1}{\mu + \sigma}\, (v, w)_H
\] 
for all $v\in V$. Using that~$w$ is also an adjoint-eigenfunction of 
$J(u^{\ast})$
and the orthogonality of Lemma~\ref{lem:leftrigthOrth}, we find that 
\begin{align*}
(\phi^{\prime}(u^*)\, v, w )_H 
&= |\lambda^{\ast} + \sigma|\ \Big[ \big(J_\sigma(u^*)^{-1} \I v, w\big)_H - \big(J_\sigma(u^*)^{-1} \I v, u^* \big)_H\, (u^*,w)_H \Big] \\
&= \frac{|\lambda^{\ast} + \sigma|}{\mu + \sigma}\, (v, w)_H. 
\end{align*}
Thus, $w$ is an adjoint-eigenfunction of~$\phi^{\prime}(u^{\ast})$ to the eigenvalue~$|\lambda^*+\sigma| / (\mu + \sigma)$. 
	
For the reverse direction let $(z,\mu)$ be an eigenpair of $\phi^{\prime}(u^*)$, i.e., for all~$v\in V$ we have 
\[
\mu\, ( z , v )_H 
= (\phi^{\prime}(u^*)\, z , v )_H
= |\lambda^{\ast} + \sigma|\ \Big[ (J_\sigma(u^*)^{-1} \I z , v)_H- (J_\sigma(u^*)^{-1} \I z, u^*)_H \, ( u^* , v)_H \Big]. 
\]
For $\mu=0$ this directly leads to $z=u^*$, since~$\mu=0$ yields the relation~$J_\sigma(u^*)^{-1} \I z = (J_\sigma(u^*)^{-1} \I z, u^*)_H\, u^*$. An application of~$J_\sigma(u^*)$ then implies that~$u^*$ and~$z$ coincide up to a multiplicative constant. Since both functions are normalized, we get $u^*=z$.
In the case $\mu\neq 0$ we employ the adjoint-eigenfunctions $v_1, v_2, \dots$ of~$J(u^{\ast})$ as test functions. This yields
\[
\tfrac{\mu}{|\lambda^{\ast} + \sigma|} \, (z, v_j)_H
= \big( J_\sigma(u^*)^{-1} \I z, v_j\big)_H - (J_\sigma(u^*)^{-1} \I z, u^*)_H\, (u^*,v_j)_H. 
\]
Note that the second term on the right-hand side vanishes due to Lemma~\ref{lem:leftrigthOrth}. The definition of being an adjoint-eigenfunction then leads to  
\[
\tfrac{\mu}{|\lambda^* + \sigma|} \, (z, v_j)_H
= \big( J_\sigma(u^*)^{-1} \I z, v_j\big)_H 
= \tfrac{1}{\lambda_j + \sigma} \, (z, v_j)_H
\]
for $j=1,2,\dots$ . Obviously, $(z, v_j)_H$ 
cannot vanish for all test functions. This implies that~$\mu = |\lambda^* + \sigma|/(\lambda_j + \sigma)$ for a certain index $j$. In other words, $\mu$ equals an eigenvalue of~$J_\sigma(u^*)^{-1} \I$ up to the factor~$|\lambda^* + \sigma|$.
\end{proof}
With the obtained knowledge about the spectrum of~$\phi^{\prime}(u^{\ast})$ we can directly apply Proposition~\ref{prop-ostrowski} to deduce an abstract local convergence result with a rate depending on the spectrum of~$J(u^*)$ relative to the implemented shift as in the matrix case. 
\begin{theorem}[abstract local convergence]
\label{theorem-local-convergence}
Consider the assumptions of Lemma~\ref{lem:local} and let $\sigma\in\R$ be a shift sufficiently close to~$-\lambda^*$. By $\mu$ we denote the (real) eigenvalue of~$J(u^*)$ which is closest to $-\sigma$ but different from~$\lambda^*$. If the shift is selected such that 
$$\rho^{\ast} := \frac{|\lambda^* + \sigma|}{|\mu + \sigma|} <1,$$
then the iterations of the $J$-method \eqref{eqn:Jmethod} are locally convergent to $u^{\ast}$ in the $V$-norm. Furthermore, for every $\eps>0$ there exists a neighborhood $S_{\eps}$ of $u^{\ast}$ and a constant $C_\eps>0$ such that
$$
\|  u^n - u^{\ast} \|_{V} \le C_\eps \hspace{2pt} |\rho^{\ast} + \eps|^n \hspace{2pt} \| u^0 - u^{\ast} \|_{V} \qquad \mbox{for all } u^0 \in S_{\eps}
$$
and $n\ge 1$.
\end{theorem}
Below, we apply this abstract result to the GPEVP of Example~\ref{exp:GPEVP}. 
%
%
\section{Quantified local convergence of the shifted $J$-method for the GPEVP}\label{sect:localGPE}
As already mentioned in the introduction, we want to apply the damped $J$-method to the GPEVP~\eqref{eqn:GPEVP:strong}, cf.~Example~\ref{exp:GPEVP}. Thus, the task is to find an eigenfunction $u^* \in H^1_0(\D)$ with $\| u^*\|^2_{L^2(\D)} = 1$ and corresponding eigenvalue $\lambda^{\ast} \in \R$ such that
\begin{align*}
-\Delta u^* + Wu^* - \Omega \mathcal{L}_z u^* + \kappa\, |u^*|^2 u^* 
= \lambda^* u^*.
\end{align*}
Recall that the potential satisfies~$W \in L^{\infty}(\D,\R)$, whereas $\kappa\ge0$ regulates the nonlinearity. Furthermore, $\mathcal{L}_z$ 
is the angular momentum operator with angular velocity $\Omega \in \R$. With these properties it is easily seen that the GPEVP can only have real eigenvalues. In the following we will also assume that  
\begin{align}
\label{assumption-W-Omega}
W(x) \ge \Omega^2 |x|^2.
\end{align}
This condition can be interpreted as that trapping frequencies are larger than the angular frequency. Physically speaking, this ensures that centrifugal forces do not become too strong compared to the strength of the trapping potential $W$. 
%
%
\subsection{The Gross-Pitaevskii energy}\label{sect:localGPE:energy}
We consider homogeneous Dirichlet boundary conditions and the short-hand notation  $L^2(\D):=L^2(\D,\C)$, $H^1_0(\D):=H^1_0(\D,\C)$, and $H^{-1}(\D):= H^{-1}(\D,\C)$ for the function spaces over $\R$, cf.~the details in Example~\ref{exp:GPEVP}. We set 
\[
V := H^1_0(\D), \qquad
H := L^2(\D), \qquad
V^* := H^{-1}(\D).
\]
This means that $\| \cdot \|$ equals the $L^2(\D)$-norm and $(\,\cdot\,,\,\cdot\,)_H := \Re(\,\cdot\,,\,\cdot\,)_{L^2(\D)}$, where $(u,v)_{L^2(\D)}:=\int_{\D} u \,\bar{v} \hspace{2pt}dx$.
Similarly, we equip $V$ with the inner product $(\,\cdot\,,\,\cdot\,)_V := \Re(\,\cdot\,,\,\cdot\,)_{H^1(\D)}$, where $(u,v)_{H^1(\D)}:=\int_{\D} \nabla u \cdot \overline{\nabla v} \hspace{2pt}dx$. 	
As before, the weak formulation is given by $\A(u^*) = A(u^*,u^*) = \lambda^*\I u^*$, where the nonlinear operator $A$ is defined in~\eqref{eqn:GPEVP:weak}. The eigenvalue problem is equivalent to finding the critical points (on the manifold associated with the $L^2$-normalization constraint) of the energy functional $E\colon V \rightarrow \R$ given by
\begin{align}
\nonumber
E(u)&:= \frac{1}{2} \int_{\D} |\nabla u|^2 + W |u|^2 - \Omega \hspace{2pt}\overline{u} \hspace{2pt} \mathcal{L}_z u + \frac{\kappa}{2} |u|^4 \dx \\
&\ = \frac{1}{2} \int_{\D} |\nablaR u|^2 + \WR |u|^2 + \frac{\kappa}{2} |u|^4 \dx, \label{eq:energy}
\end{align}
cf.~\cite{DaK10}. Recalling that $\WR(x)=W(x) - \tfrac{1}{4}\, \Omega^2\, |x|^2$ and considering assumption \eqref{assumption-W-Omega}, we see that the energy functional is bounded from below by a positive constant $c=c(\Omega,\D)$, i.e.,  
$$
E(u) \ge c > 0  \qquad \mbox{ for all } u\in V \mbox{ with } \| u \| = 1.
$$
Hence, $E$ is weakly lower semi-continuous and bounded from below, which yields the existence of a minimizer that is typically called a ground state. 

If $\Omega=0$, i.e., in the absence of a rotating potential, then the GPEVP has infinitely many eigenvalues $0<\lambda_1^{\ast}<\lambda_2^{\ast}\le\lambda_3^{\ast} \le \cdots < \infty$, where the ground state eigenvalue $\lambda_1^{\ast}$ is simple, cf.~\cite{CCM10,HenP18ppt}. If $\Omega\not=0$, the smallest eigenvalue is typically no longer simple \cite{Bao-et-al-2005}. This case refers to the physical phenomenon of a broken symmetry, where the ground state can have different shapes which differ in their number and location of vortices.
%
%
\subsection{The $J$-operator for the GPEVP}\label{sect:localGPE:Joperator}
In order to formulate the $J$-method for the GPEVP, we need to compute the linearization $J(u)$ according to \eqref{abstract-definiton-J}. Hence, by calculating the Fr\'echet derivative of $\mathcal{A}(u)=A(u,u)$ using \eqref{eqn:GPEVP:weak}, we obtain
\begin{eqnarray}\label{def-J-real}
\langle J(u) v , w  \rangle = \Re\langle \mathcal{J}(u) v , w  \rangle,
\end{eqnarray}
where 
\begin{multline}\label{def-J-full}
\langle \mathcal{J}(u) v , w  \rangle
:= \int_{\D} \nablaR v \cdot \overline{\nablaR w} + \WR v \overline{w} \dx\\
+ \frac{\kappa}{\| u \|^2} \int_{\D} \left( u \overline{v} + 2 v \overline{u} \right) u \overline{w} \dx - \frac{2\kappa\, 
\Hscapro{u}{v} }{\| u \|^4} \int_{\D} |u|^2 u \overline{w} \dx.
\end{multline}
Note that the operator~$J(u)$ induces an $\R$-bilinearform $ \langle J(u)\,\cdot,\, \cdot\,\rangle$, i.e., we still consider~$V$ as an~$\R$-vector space and have bilinearity only for multiplicative constants in~$\R$. 
Recall that we consider the real-Fr\'echet derivative, which also fits in the standard framework used in quantum mechanics. Moreover, the complex-Fr\'echet derivative does not exist for the present example.

We will show coercivity of $J(u)$ up to a shift in Lemma \ref{lem:Jcoercive} below. Before that, it is worth to mention that the eigenvalue problem can be equivalently expressed in terms of $ \mathcal{J}(u)$, which also yields the typical structure with standard inner products. We have the following result.
\begin{proposition}
Consider the GPEVP with full operator $\mathcal{J}$ given by \eqref{def-J-full} and its real part~$J$ given by \eqref{def-J-real}. Then $\lambda^{\ast} \in \R$ is an eigenvalue with $L^2$-normalized eigenfunction $u^{\ast}\in V$, i.e., 
$$
\langle J(u^{\ast}) , v \rangle 
= \lambda^{\ast}\, \Hscapro{ u^{\ast} }{ v }\qquad \mbox{for all } v\in V
$$
if and only if 
$$
\langle \mathcal{J}(u^{\ast}) , v \rangle = \lambda^{\ast} (u^{\ast}, v)_{L^2(\D)} \qquad \mbox{for all } v\in V.
$$
\end{proposition}
\begin{proof}
If $(\lambda^{\ast},u^{\ast}) \in \R \times V$ solves $\langle \mathcal{J}(u^{\ast}) , v \rangle = \lambda^{\ast} ( u^{\ast} , v)_{L^2(\D)}$, then we can take the real part on both sides and obtain $\langle J(u^{\ast}) , v \rangle = \lambda^{\ast} \Hscapro{ u^{\ast} }{ v }$. 
Vice versa, if $(\lambda^{\ast},u^{\ast}) \in \R \times V$ solves $\Re \langle  \mathcal{J}(u^{\ast}) , v \rangle =  \lambda^{\ast}\, \Re(u^{\ast} , v)_{L^2(\D)}$, then we can use test functions of the form $\ci v \in V$ to obtain $\Im \langle  \mathcal{J}(u^{\ast}) , v \rangle =  \lambda^{\ast}\, \Im(u^{\ast} , v)_{L^2(\D)}$. Multiplying the second equation with the complex number $\ci$ and adding it to the first equation readily yields
$\langle \mathcal{J}(u^{\ast}) , v \rangle = \lambda^{\ast} (u^{\ast}, v)_{L^2(\D)}$.
\end{proof}
The proposition shows that we can either interpret the eigenvalue problem with real parts only (i.e., with $J$) or with the full operator $\mathcal{J}$. With analogous arguments, we can also prove that $J(u^{\ast})^{-1} \I = \mathcal{J}(u^{\ast})^{-1} \I^\C$, where 
\begin{align}
\label{I-full}
\I^\C v:=(v ,\, \cdot\, )_{L^2(\D)}.
\end{align}
This justifies that we can interpret the iterations of the damped $J$-method \eqref{eqn:JmethodDamped} equivalently with~$J$ or~$\mathcal{J}$. In particular, we have the following conclusion. 
\begin{conclusion}
Consider the GPEVP with full operator $\mathcal{J}$ given by \eqref{def-J-full} and its real part $J$ defined in~\eqref{def-J-real}. Given $u^n\in V$ with $\|u^n\|=1$ and a step size $\tau_n>0$, we assume that $J_\sigma(u^n)^{-1}$ is invertible. Then the iterations of the $J$-method \eqref{eqn:JmethodDamped} can be equivalently characterized by the $\mathcal{J}$-iteration
\begin{align}
\label{eqn:JmethodDamped-new}
\tilde{u}^{n+1} 
= (1-\tau_n)u^n + \tau_n\, \gamma_n^\C\, \mathcal{J}_\sigma(u^n)^{-1} \I^\C \hspace{1pt} u^n, \qquad 
u^{n+1} = \frac{\tilde{u}^{n+1}}{\| \tilde{u}^{n+1} \|}
\end{align} 
with $\mathcal{J}_\sigma(u) := \mathcal{J}(u)+\sigma \I^\C$ and~$(\gamma_n^\C)^{-1} := (\mathcal{J}_\sigma(u^n)^{-1} \I^\C\hspace{1pt} u^n , u^n)_{L^2(\D)}$. The assumed existence of $J_\sigma(u^n)^{-1} \I$ implies the existence (and uniqueness) of $\mathcal{J}_\sigma(u^n)^{-1} \I^\C$.
\end{conclusion}
%
%
\subsection{Local convergence}\label{sect:localGPE:conv}
Finally, we apply Theorem \ref{theorem-local-convergence} to the GPEVP. 
\begin{theorem}[quantified convergence for the GPEVP]
\label{corollary-quantified-convergence}
Consider the GPEVP as described in Example \ref{exp:GPEVP} and let $u^n$ denote the iterates generated by the shifted $J$-method (without damping), i.e., 
\begin{align*}
u^{n+1}
= \phi(u^n)
= \frac{J_\sigma(u^n)^{-1} \I u^n}{\|  J_\sigma(u^n)^{-1} \I u^n \|}.
\end{align*}
By $u^{\ast}\in V=H^1_0(\D)$ we denote an $L^2$-normalized eigenfunction to \eqref{eqn:GPEVP:strong} with eigenvalue~$\lambda^{\ast}$. Assume that $\lambda^{\ast}$ is a simple eigenvalue of $J(u^{\ast})$ and that the shift $\sigma \not= - \lambda^{\ast}$ is such that $J_{\sigma}(u^{\ast})$ has a bounded inverse.  
Further, let the shift~$\sigma$ be sufficiently close to~$-\lambda^{\ast}$ and let $\mu$ be the eigenvalue of $J_{\sigma}(u^{\ast})$ so that $|\mu + \sigma|$ is minimal and 
$$
\frac{|\lambda^{\ast} + \sigma |}{| \mu + \sigma|} < 1.
$$
Then there exists a neighborhood $S \subset V$ of $u^{\ast}$ such that for all starting values $u^0 \in S$ we have
$$
\lim_{n \rightarrow \infty} \| u^n - u^{\ast} \|_{V} = 0. 
$$
Furthermore, for every $\eps>0$ there is a neighborhood $S_{\eps} \subset V$ of $u^{\ast}$ and a constant $C(\eps)>0$ such that for all $u^0 \in S_{\eps}$ (and $n\ge 1$) it holds
\begin{align*}
\| u^{n} - u^{\ast} \|_V\ 
\le\ C(\eps) \left( \frac{|\lambda^{\ast} + \sigma |}{| \mu + \sigma|} + \eps \right)^n \| u^{0} - u^{\ast} \|_V.
\end{align*}
Hence, if the shift $\sigma$ is close to~$-\lambda^{\ast}$ and $u^n$ close to~$u^*$, then we have locally a linear convergence with rate
$$
\frac{|\lambda^{\ast} + \sigma |}{| \mu + \sigma|} + \eps < 1.
$$
for any $\eps>0$.
\end{theorem}
\begin{proof}
By definition of the problem (cf.~Example~\ref{exp:GPEVP}) all eigenvalues are real.
Furthermore, the spectrum is countable and does not have an accumulation point in $\C$. This is a direct consequence from the observation that $J_\sigma(u^{\ast})^{-1} \I\colon L^2(\D) \rightarrow L^2(\D)$ is a compact operator for all shifts $-\sigma$ that are not in the spectrum of $J(u^{\ast})$ (i.e., we have compact resolvents). Hence, we can apply Theorem \ref{theorem-local-convergence} which readily proves the result.
\end{proof}
In the finite dimensional case based on a finite difference discretization, a corresponding result was presented in \cite{JarKM14}.
Motivated by the above convergence result, a practical realization of the iterations can be based on the more natural formulation of the $J$-method given by \eqref{eqn:JmethodDamped-new}. This is also the version for which we discuss the implementation in Appendix~\ref{app:Jmatrix}. 
%
%
\section{Global convergence of the damped $J$-method for the GPEVP}\label{sect:GPE}
In this section we come back to the question of invertibility of the operator~$J$ (which hence also implies invertibility of $\mathcal{J}$). This will then lead to a globally convergent method. 
%
\subsection{Coercivity of the shifted $J$-operator}\label{sect:GPE:coercive}
We first show that the operator $J$ is -- up to a shift -- coercive. 
\begin{lemma}
\label{lem:Jcoercive}
Given $u\in V$ and assumption~\eqref{assumption-W-Omega}, the 
operator~$J(u)$ corresponding to the Gross-Pitaevskii operator satisfies a G\aa{}rding inequality. More precisely, for any~$\sigma \ge \frac{\kappa}{3} \| u \|_{L^4(\D)}^4 / \| u \|^4$ the bilinear form 
$$
\langle (J(u)+\sigma\I)\hspace{3pt} \cdot \hspace{2pt} , \hspace{2pt} \cdot \hspace{2pt} \rangle\colon V \times V \rightarrow \R
$$ 
is coercive and thus, the operator $J_\sigma(u) = J(u)+\sigma\I\colon V\to V^*$ is invertible.
\end{lemma}
\begin{proof}
Consider $v\in V$. We start with considering the rotational gradient, for which we observe with Young's inequality that
\begin{eqnarray*}
\int_{\D} | \nablaR v|^2 \dx  
\ge \frac{1}{2} \int_{\D} |\nabla v|^2 \dx - \frac{3}{4} \Omega^2 \int_{\D}  |x|^2 |v|^2 \dx.
\end{eqnarray*}
Hence, with $\WR= W(x) - \tfrac{1}{4}\, \Omega^2|x|^2$ and assumption \eqref{assumption-W-Omega} we have
\begin{eqnarray}
\label{lower-bound-rot-grad-norm}
\int_{\D} | \nablaR v|^2 + \WR |v|^2 \dx  \ge \frac{1}{2} \int_{\D} |\nabla v|^2 \dx.
\end{eqnarray}
This leads to 
\begin{eqnarray}
\label{coercivity-J-step1}
\nonumber\lefteqn{
\langle J(u) v , v  \rangle = \int_{\D} |\nablaR v|^2 + \WR |v|^2 \dx  } \\
\nonumber&\enspace& \quad + \frac{\kappa}{\| u \|^2} \int_{\D} |u|^2 |v|^2 \dx
+ \frac{2\kappa}{\| u \|^2} \int_{\D}  |\Re( u \overline{v})|^2 \dx
- 2\kappa \frac{ \Hscapro{u}{v} }{\| u \|^4}  \int_{\D}  |u|^2\, \Re(u \overline{v}) \dx \\
\nonumber&\overset{\eqref{lower-bound-rot-grad-norm}}{\ge}& 
\frac{1}{2}\, \| \nabla v \|^2 +  \frac{\kappa}{\| u \|^2} \int_{\D} |u \overline{v}|^2 \dx
+ \frac{2\kappa}{\| u \|^2} \int_{\D}  |\Re( u \overline{v})|^2 \dx
- 2\kappa \frac{\Hscapro{u}{v} }{\| u \|^4}  \int_{\D}  |u|^2\, \Re(u \overline{v}) \dx\\
\nonumber&\ge& 
\frac{1}{2}\, \| \nabla v \|^2 + 3\frac{\kappa}{\| u \|^2} \int_{\D}  |\Re(u \overline{v})|^2 \dx
- 2\kappa \frac{ \Hscapro{u}{v} }{\| u \|^4}  \int_{\D}  |u|^2\, \Re(u \overline{v}) \dx.\\
\end{eqnarray}
To estimate the negative part, we apply once more Young's inequality with some parameter~$\mu>0$ and the Cauchy Schwarz inequality to get
\begin{align*}
2\, \Hscapro{u}{v} \int_{\D} |u|^2\, \Re (u \overline{v} )\dx
&\le \frac1\mu\, |\Hscapro{u}{v}|^2 + \mu\, \Big( \int_{\D} |u|^2  \Re(u \overline{v} ) \dx \Big)^2 \\
&\le \frac1\mu\, \| u \|^2 \| v \|^2 + \mu\, \|u\|^4_{L^4(\D)}  \int_{\D}  |\Re (u \overline{v} )|^2 \dx.
\end{align*}
Using this estimate in \eqref{coercivity-J-step1} with the particular choice $\mu = 3\, \|u\|^2/ \|u\|^4_{L^4(\D)}$, we obtain 
\begin{align*}
\langle J(u) v , v  \rangle 
&\ge \frac{1}{2}\, \| \nabla v \|^2 + \frac{3\kappa}{\| u \|^2} \int_{\D} |\Re(u \overline{v})|^2 \dx
-  \frac{\kappa}{3} \frac{ \|u\|^4_{L^4(\D)} \| v \|^2 }{\| u \|^4}  - \frac{3\kappa}{\| u \|^2} \int_{\D}  |\Re (u \overline{v} )|^2 \dx\\
&= \frac{1}{2}\, \| \nabla v \|^2 -  \frac{\kappa}{3}\frac{\| u \|_{L^4(\D)}^4}{\| u \|^4}  \| v \|^2.
\end{align*}
Thus, we conclude
\[
\langle J(u) v , v  \rangle 
\ge \frac 12\, \| v \|_V^2 -  \sigma\, \| v \|^2
= \frac{1}{2}\, \| v \|_V^2 -  \sigma\, \langle \I v , v \rangle.
\qedhere
\]
\end{proof}
Recall the definition of the energy $E$ in Section~\ref{sect:localGPE:energy}. Motivated by the previous lemma, we also define the shifted energy $E_\sigma(u) := E(u) + \frac 12 \sigma \|u\|^2$ such that $E_0(u) = E(u)$. For $u\in V$ with normalization constraint $\| u \| = 1$, a sufficient shift in the sense of Lemma~\ref{lem:Jcoercive} is thus given by 
\[
\sigma 
:= \frac 43\, E(u) 
\ge \frac \kappa3\, \| u \|_{L^4(\D)}^4. 
\]
Note that for a normalized function $u$ we can express the Rayleigh quotient in terms of the energy by
$$
\lambda(u) 
:= \langle \A(u), u\rangle
= 2 E(u) + \frac{\kappa}{2}\, \| u \|_{L^4(\D)}^4.
$$
In particular, this formula relates the eigenvalues with the energies of the eigenfunctions.
%
\subsection{Feasibility of the $J$-method}
For the feasibility of the damped $J$-method, we need to guarantee a priori that $J_\sigma(u^n)$ stays invertible throughout the iteration process. We fix the shift $\sigma$ in the beginning of the iteration, e.g., by $\sigma := \frac 43  E(u^0)$. Now, the aim is to show that the energy of the iterates does not increase such that $\sigma \ge \frac 43  E(u^n)$ for all $n\ge0$. 
\begin{lemma}
Recall the shifted energy $E_\sigma(u) := E(u) + \frac 12 \sigma \|u\|^2$ with $E(u)$ being defined in \eqref{eq:energy}. Further recall the definition of the $J$-method~\eqref{eqn:JmethodDamped} with iteration steps~$\tilde{u}^{n+1} = (1-\tau)u^n + \tau\, \gamma_n\, J_\sigma(u^n)^{-1} \I u^n$.
Then, for any step size $\tau \le \frac 12$ we have the following guaranteed estimate for the difference of the shifted energies
\begin{align}
\notag
E_{\sigma}(u^n) - E_{\sigma}(\tilde{u}^{n+1}) 
\ge &- \frac{3\kappa}{2} \int_{\D}  |\tilde{u}^{n+1} - u^{n}|^4 \dx  \\
&+\, (\tfrac{1}{\tau}-\tfrac 12) \int_{\D} |\nablaR (\tilde{u}^{n+1}-u^n)|^2 + (\sigma+\WR) |(\tilde{u}^{n+1}-u^n)|^2 \dx.  \label{eqn:energy-difference-estimate}
\end{align}
\end{lemma}
\begin{proof}
We start by establishing a couple of identities for different evaluations of $J$. Here we exploit the $L^2$-orthogonality~\eqref{eqn:L2-orthogonality}, i.e.,  $(u^n,\tilde{u}^{n+1}-u^n)=0$. Together with $\| u^n \|=1$ this implies $(u^n,\tilde{u}^{n+1})=1$. Using these facts, we observe that
\begin{align}
\nonumber\langle J(u^n) \tilde{u}^{n+1} , \tilde{u}^{n+1}  \rangle 
\nonumber&= \int_{\D} |\nablaR \tilde{u}^{n+1}|^2 + \WR |\tilde{u}^{n+1}|^2 \dx + \kappa \int_{\D} |\tilde{u}^{n+1}|^2 \, |u^n|^2  \dx\\
\nonumber&\qquad+
2 \kappa \int_{\D} \left( \Re ( u^n \, \overline{\tilde{u}^{n+1}} ) - |u^n|^2 \right) 
\Re(u^n \, \overline{\tilde{u}^{n+1}}) \dx \\
\nonumber&= 2E(\tilde{u}^{n+1}) + \kappa \int_{\D} |\tilde{u}^{n+1}|^2 \, \big( |u^n|^2-  |\tilde{u}^{n+1}|^2 \big) \dx + \frac\kappa2 \int_{\D} |\tilde{u}^{n+1}|^4 \dx \\
\label{identity-Jun-tildeu}&\qquad + 2 \kappa \int_{\D} \big( \Re ( u^n \, \overline{\tilde{u}^{n+1}}) - |u^n|^2 \big)\, \Re( u^n \, \overline{\tilde{u}^{n+1}}) \dx.
\end{align}
Next, using again the $L^2$-orthogonality together with the definition of $\tilde{u}^{n+1}$ in~\eqref{eqn:JmethodDamped}, we have 
\begin{eqnarray}
\label{pseudo-energy-identity}
\nonumber\lefteqn{\tfrac{1}{\tau} \langle J_{\sigma}(u^n) (\tilde{u}^{n+1} - u^n)  , \tilde{u}^{n+1} - u^n \rangle} \\
\nonumber&=& -  \langle J_{\sigma}(u^n) u^n , \tilde{u}^{n+1} - u^n \rangle + \gamma^n  \langle J_{\sigma}(u^n) J_{\sigma}(u^n)^{-1}\I  u^n , \tilde{u}^{n+1} - u^n \rangle \\
\nonumber&=& -  \langle J_{\sigma}(u^n) u^n , \tilde{u}^{n+1} - u^n \rangle + \gamma^n ( u^n , \tilde{u}^{n+1} - u^n ) \\
&=& -  \langle J_{\sigma}(u^n) u^n , \tilde{u}^{n+1} - u^n \rangle.
\end{eqnarray}
With this equality, we conclude 
\begin{eqnarray}
\nonumber\langle J_{\sigma}(u^n) u^n , u^n  \rangle 
\nonumber&=&  \langle J_{\sigma}(u^n) u^n , \tilde{u}^{n+1} \rangle  - \langle J_{\sigma}(u^n) u^n , \tilde{u}^{n+1} - u^n  \rangle \\
\nonumber&\overset{\eqref{pseudo-energy-identity}}{=}&  \langle J_{\sigma}(u^n) u^n , \tilde u^{n+1} \rangle  + \tfrac{1}{\tau} \langle J_{\sigma}(u^n) (\tilde{u}^{n+1} - u^n) , \tilde{u}^{n+1} - u^n  \rangle \\
\nonumber&=&  \langle J_{\sigma}(u^n)  \tilde{u}^{n+1} ,  \tilde{u}^{n+1} \rangle - \langle J_{\sigma}(u^n) ( \tilde{u}^{n+1} - u^n) ,  u^{n} \rangle \\
\label{proof-en-dissp-s4}&&\qquad + (\tfrac{1}{\tau}-1) \langle J_{\sigma}(u^n) (\tilde{u}^{n+1} - u^n) , \tilde{u}^{n+1} - u^n  \rangle. 
\end{eqnarray}
Here we need to have a closer look at the term $\langle J_{\sigma}(u^n) ( \tilde{u}^{n+1} - u^n) ,  u^{n} \rangle$. By the definition of $J(u^n)$ it is easily seen that
\begin{align}
\label{symmetry-fix-J}
\langle J(u^n) ( \tilde{u}^{n+1} - u^n),  u^{n}  \rangle - \langle J(u^n) u^n ,  \tilde{u}^{n+1} - u^n  \rangle
= 2 \kappa \int_{\D} |u^n|^2 \Re \left( ( \tilde{u}^{n+1} - u^n) \overline{u^n} \right) \dx. 
\end{align}
Plugging this into \eqref{proof-en-dissp-s4} yields for the shifted energy 
\begin{eqnarray*}
\lefteqn{2E_{\sigma}(u^n)} \\
&=& \langle J_{\sigma}(u^n) u^n, u^{n} \rangle - \frac{\kappa}{2}\int_{\D}|u^n|^4 \dx\\
&\overset{\eqref{proof-en-dissp-s4}}{=}& \langle J_{\sigma}(u^n)  \tilde{u}^{n+1} ,  \tilde{u}^{n+1} \rangle - \langle J_{\sigma}(u^n) ( \tilde{u}^{n+1} - u^n) ,  u^{n} \rangle \\
&&\quad + (\tfrac{1}{\tau}-1) \langle J_{\sigma}(u^n) (\tilde{u}^{n+1} - u^n), \tilde{u}^{n+1} - u^n  \rangle - \frac{\kappa}{2}\int_{\D}|u^n|^4 \dx \\
&\overset{\eqref{pseudo-energy-identity}, \eqref{symmetry-fix-J}}{=}& \langle J_{\sigma}(u^n)  \tilde{u}^{n+1} ,  \tilde{u}^{n+1} \rangle 
+\tfrac{1}{\tau} \langle J_{\sigma}(u^n) (\tilde{u}^{n+1} - u^n) , \tilde{u}^{n+1} - u^n  \rangle
- \frac{\kappa}{2}\int_{\D}|u^n|^4 \dx \\
&&\quad + (\tfrac{1}{\tau}-1) \langle J_{\sigma}(u^n) (\tilde{u}^{n+1} - u^n), \tilde{u}^{n+1} - u^n  \rangle 
- 2 \kappa \int_{\D} |u^n|^2 \Re \left( ( \tilde{u}^{n+1} - u^n) \overline{u^n} \right) \dx \\
&\overset{\eqref{identity-Jun-tildeu}}{=}& 
2E_{\sigma}(\tilde{u}^{n+1}) - \kappa \int_{\D} |\tilde{u}^{n+1}|^2 \, ( |\tilde{u}^{n+1}|^2 - |u^n|^2 ) \dx
+ \frac{\kappa}{2}\int_{\D}|\tilde{u}^{n+1}|^4 \dx
- \frac{\kappa}{2}\int_{\D}|u^n|^4 \dx\\
&\enspace&\quad
-2 \kappa \int_{\D} |u^n|^2 \Re \left( ( \tilde{u}^{n+1} - u^n) \overline{u^n} \right) \dx +\,2 \kappa \int_{\D} \big( \Re \big( u^n \, \overline{\tilde{u}^{n+1}} \big) - |u^n|^2 \big)\, 
\Re( u^n \, \overline{\tilde{u}^{n+1}}) \dx \\
&\enspace&\quad+\, (\tfrac{1}{\tau}-1) \langle J_{\sigma}(u^n) (\tilde{u}^{n+1} - u^n) , \tilde{u}^{n+1} - u^n  \rangle 
+  \tfrac{1}{\tau} 
\langle J_{\sigma}(u^n) (\tilde{u}^{n+1} - u^n) , \tilde{u}^{n+1} - u^n  \rangle.
\end{eqnarray*}
Since 
$\Re\, z = \Re\, \overline{z}$ for all $z\in\C$, we conclude that
\begin{eqnarray}
\nonumber\lefteqn{2E_{\sigma}(u^n) - 2E_{\sigma}(\tilde{u}^{n+1})} \\
\nonumber &=&
- \kappa \int_{\D} |\tilde{u}^{n+1}|^2 \, ( |\tilde{u}^{n+1}|^2 - |u^n|^2 ) \dx
+ \frac{\kappa}{2}\int_{\D}|\tilde{u}^{n+1}|^4 \dx
- \frac{\kappa}{2}\int_{\D}|u^n|^4 \dx\\
\label{energy-difference-equality} &\enspace&\qquad 
+ 2 \kappa \int_{\D} \left( |u^n|^2 - \Re ( \tilde{u}^{n+1} \overline{u^n} )\right)^2 \dx\ +\, (\tfrac{2}{\tau}-1) \hspace{2pt}\langle J_{\sigma}(u^n) (\tilde{u}^{n+1} - u^n) , \tilde{u}^{n+1} - u^n  \rangle.
\end{eqnarray}
Using that
\begin{eqnarray}
\nonumber\lefteqn{\langle J(u^n) (\tilde{u}^{n+1} - u^n) , (\tilde{u}^{n+1} - u^n)  \rangle= \int_{\D} |\nablaR (\tilde{u}^{n+1} - u^n)|^2 + \WR\, |\tilde{u}^{n+1} - u^n|^2 \dx} \\
\label{J-error-error} &\enspace&\qquad+ \kappa \int_{\D} |\Re ( u^n (\overline{\tilde{u}^{n+1} - u^n}) )|^2
+ |u^n|^2 \, |\tilde{u}^{n+1} - u^n|^2 \dx,\hspace{50pt}
\end{eqnarray}
and $\Re(u^n (\overline{\tilde{u}^{n+1} - u^n})) = \Re(\tilde{u}^{n+1}\overline{u^n} ) - |u^n|^2$
we see that 
\begin{eqnarray*}
\lefteqn{2E_{\sigma}(u^n) - 2E_{\sigma}(\tilde{u}^{n+1})} \\
&\overset{\eqref{energy-difference-equality},\eqref{J-error-error}}{\ge}&  
- \kappa \int_{\D} |\tilde{u}^{n+1}|^2 ( |\tilde{u}^{n+1}|^2 - |u^n|^2 ) \dx 
+ \frac{\kappa}{2}\int_{\D}|\tilde{u}^{n+1}|^4 \dx -\, \frac{\kappa}{2}\int_{\D}|u^n|^4 \dx \\
&\enspace&\qquad+\, (\tfrac{2}{\tau}-1) \int_{\D} |\nablaR (\tilde{u}^{n+1}-u^n)|^2 + \sigma\,  |(\tilde{u}^{n+1}-u^n)|^2 + \WR\, |(\tilde{u}^{n+1}-u^n)|^2\dx \\
&\enspace&\qquad + \kappa\, (\tfrac{2}{\tau}-1) \int_{\D} |u^n|^2 | \tilde{u}^{n+1} - u^n|^2 \dx . 
\end{eqnarray*}
Finally, an application of the triangle and Young's inequality yields the estimate
\begin{align*}
- 2 \int_{\D} |\tilde{u}^{n+1}|^2 (|\tilde{u}^{n+1}|^2 - &|u^n|^2) \dx 
- \int_{\D}|u^n|^4 \dx + \int_{\D}|\tilde{u}^{n+1}|^4 \dx \\
&= - \int_{\D} |\tilde{u}^{n+1}+ u^{n}|^2  |\tilde{u}^{n+1} - u^{n}|^2 \dx\\
&\ge - \int_{\D} 3\, |\tilde{u}^{n+1} - u^{n}|^4 + 6\, |u^n|^2 |\tilde{u}^{n+1} - u^{n}|^2 \dx,  
\end{align*}
which then implies the desired inequality if $\tau\le \frac{1}{2}$.
\end{proof}
With this result we are now in the position to prove uniform boundedness of the energy of the iterates and even a reduction of the energy. 
\begin{theorem}[energy reduction]\label{theorem-energy-reduction}
Given $u^0 \in V$ with $\|u^0\|=1$, we let $\sigma \ge \frac 43 E(u^0)$ and $W(x) \ge \Omega^2 |x|^2$. Then, there exists a $\tau^{\ast}>0$ (that only depends on $u^0$ and its energy) such that for all $0 < \tau_n \le \tau^{\ast}$ 
the sequence obtained by the damped $J$-method~\eqref{eqn:JmethodDamped} 
is well-posed and
strictly energy diminishing for all $n$, i.e., it holds
\[ E(u^{n+1}) \le E(u^{n}) \le E(u^0), \]
where $E(u^{n+1}) < E(u^{n})$ if $u^{n}$ is not already a critical point of $E$ and thus an eigenstate.
\end{theorem}
\begin{proof}
We proceed inductively. From estimate \eqref{lower-bound-rot-grad-norm} in the proof of Lemma~\ref{lem:Jcoercive} we know that $J_{\sigma}(u^0)$ is coercive with constant $1/2$, i.e., $\|\nabla v\|^2 = \| v \|_{V}^2 \le 2\, \langle J_{\sigma}(u^0)v, v \rangle $. 
Together with \eqref{pseudo-energy-identity} and the $L^2$-orthogonality~\eqref{eqn:L2-orthogonality}, this implies 
\begin{align*}
\frac 1 {2\tau_0} \int_{\D} &|\nabla \tilde{u}^1 - \nabla u^{0} |^2 \dx \\
&\quad\le -\, \langle J_{\sigma}(u^0) u^0 , \tilde{u}^{1} - u^0 \rangle \\
&\quad= \Re \left( \int_{\D} \nablaR u^0 \cdot \overline{\nablaR (u^0 - \tilde{u}^{1})} 
+ \WR\, u^0 \overline{(u^0 - \tilde{u}^{1})} \dx 
+ \kappa \int_{\D} |u^0|^2 \hspace{2pt} u^0 \, \overline{( u^0 - \tilde{u}^{1})} \dx \right) \\
&\quad \le
\| \nablaR u^0 \| \hspace{2pt} \| \nablaR (\tilde{u}^{1}-u^0) \|
+ \| \sqrt{\WR} u^0 \| \hspace{2pt} \| \sqrt{\WR} (\tilde{u}^{1}-u^0) \|
+ \kappa \|  u^0 \|_{L^6(\D)}^3 \| \tilde{u}^{1}-u^0 \| \\
&\quad \le C \sqrt{E(u^0)} \hspace{2pt}  \| \nabla (\tilde{u}^{1}-u^0) \|.
\end{align*}
Hence, we get
$$
\| \nabla \tilde{u}^{1} - \nabla u^0\|^2
= \int_{\D} | \nabla \tilde{u}^{1} - \nabla u^0|^2 \dx 
\le 4\, \tau_0^2 \hspace{2pt} C^2 E(u^0)
=: \tau_0^2\, C_0.
$$
Assume that $\tau_0\le \tau^{\ast}\le 2$, where $\tau^{\ast}$ is selected sufficiently small compared to $C_0$. Then we have that $\| \nabla \tilde{u}^{1} - \nabla u^0\|<1$ and the Sobolev embedding of $L^4(\D) \hookrightarrow H^1_0(\D)$ with constant~$C_S$ implies that
$$
 \int_{\D} | \tilde{u}^{1} -  u^0|^4 \dx 
 \le C_S  \int_{\D} | \nabla \tilde{u}^{1} - \nabla u^0|^2 \dx. 
$$
Consequently, we can use \eqref{eqn:energy-difference-estimate} and \eqref{lower-bound-rot-grad-norm} to observe that the energy difference fulfills
\begin{eqnarray}
\nonumber\lefteqn{E_{\sigma}(u^0) - E_{\sigma}(\tilde{u}^{1})}\\ 
\nonumber&\ge& - \frac{3\kappa}{2} \int_{\D} |\tilde{u}^{1}-u^0|^4 \dx 
+ \tfrac 12(\tfrac{1}{\tau_0}-\tfrac 12) \int_{\D} |\nabla (\tilde{u}^{1}-u^0)|^2 \dx\\
\nonumber&\enspace&\qquad + \sigma (\tfrac{1}{\tau_0}-\tfrac 12) \int_{\D} |\tilde{u}^{1}-u^0|^2 \dx \\
\label{eqn:inProofEnergyDifference}&\ge& 
\tfrac 12(\tfrac{1}{\tau_0}-\tfrac 12-3\kappa\, C_S) \int_{\D} |\nabla (\tilde{u}^{1}-u^0)|^2 \dx
+ \sigma (\tfrac{1}{\tau_0}-\tfrac 12) \int_{\D} |\tilde{u}^{1}-u^0|^2 \dx.
\end{eqnarray}
Hence, if 
$$
\tau_0 \le \tau^{\ast} 
< \min \Big\{ \frac{2}{1 + 6 \kappa C_S},\, C_0^{-1/2},\, \frac 12 \Big\}, 
$$
then we have
\begin{eqnarray*}
E_{\sigma}(u^{1} )
\le E_{\sigma}(\tilde{u}^{1} )
\le	E_{\sigma}(u^0),
\end{eqnarray*}
where we have used that the iterations increase the mass intermediately, i.e., $\| \tilde{u}^{1} \| \ge 1$ as shown in \eqref{intermediate-mass-growth}. Note that since $u^0$ and $u^1$ are normalized in $L^2(\D)$, we can drop the shift, leading to $E(u^{1} ) \le E(u^0)$. Hence, we have $\sigma \ge \frac{4}{3}E(u_0) \ge \frac{4}{3}E(u_1)$ and Lemma~\ref{lem:Jcoercive} guarantees that $J_{\sigma}(u^1)$ is still coercive. 
Inductively, we can repeat the arguments for $u^n$ with the same generic constant~$C_0$ to show that for $\tau^n \le \tau^{\ast}$ we have
$$
E_{\sigma}(u^{n+1}) \le E_{\sigma}(u^{n}). 
$$
Since the energy is diminished in every iteration, the coercivity of 
$\langle J_\sigma(u^n) \hspace{2pt}\cdot\hspace{2pt}, \hspace{2pt}\cdot\hspace{2pt} \rangle$ is maintained and all the iterations are well-defined. Finally, we note that because of~\eqref{eqn:inProofEnergyDifference} we have $E_{\sigma}(u^n) = E_{\sigma}(\tilde{u}^{n+1} )$ if and only if $u^n=\tilde{u}^{n+1}$. However, this can only happen if 
\begin{align*}
J_\sigma(u^n) u^{n} = \, \gamma_n\, \I u^n,
\end{align*} 
i.e., if $u^n$ is already an eigenfunction with eigenvalue $\lambda=\gamma_n$.
\end{proof}
It is interesting to note that the $L^2$-norm of $\tilde{u}^{n}$ cannot diverge. We see this in the following conclusion.
\begin{conclusion}\label{conc-tildeu}
In the setting of Theorem \ref{theorem-energy-reduction} it holds that~$\| \tilde{u}^{n} \| \rightarrow 1$ for $n\rightarrow \infty$.
\end{conclusion}

\begin{proof}
In the proof of Theorem \ref{theorem-energy-reduction} we have seen that 
$$
E_{\sigma}(u^n) - E_{\sigma}(u^{n+1} ) 
\ge  \tfrac{1}{2}(\tfrac{1}{\tau^{\ast}}-\tfrac{1}{2} - 3\, \kappa\, C_S) \int_{\D} |\nabla (\tilde{u}^{n+1}-u^n)|^2 \dx.
$$
Since $E_{\sigma}(u^n)$ is monotonically decreasing and bounded from below, we have $E_{\sigma}(u^n) - E_{\sigma}(u^{n+1} ) \rightarrow 0$. This together with the Poincar\'e-Friedrichs inequality implies that
$$
\| \tilde{u}^{n+1} \| \le \| \tilde{u}^{n+1} - u^{n} \| + \| u^n \| \rightarrow 1
$$
for $n\rightarrow \infty$.
\end{proof}
%
\subsection{Convergence and optimal damping} 
In this subsection we prove the convergence of the $J$-method for a suitable choice of damping parameters. We can make practical use of this result by selecting $\tau_n$ in each iteration step by the minimizer of a simple one-dimensional minimization problem.
\begin{theorem}[global convergence]
\label{theorem-global-convergence}
Suppose that the assumptions of Theorem~\ref{theorem-energy-reduction} are fulfilled. Additionally assume that $\tau_n$ is selected such that it does not degenerate, i.e., is uniformly bounded away from zero. Then there exists a limit energy 
$E^{\ast}:=\lim_{n\rightarrow \infty} E(u^{n})$
and, up to a subsequence, we have that the iterates $u^n$ of the damped $J$-method converge strongly in $H^1(\D)$ to a limit $u^{\ast}\in V$. The limit is an $L^2$-normalized eigenfunction with some eigenvalue $\lambda^{\ast}>0$, i.e.,  
$$
 \mathcal{A}(u^{\ast})  =  \lambda^{\ast} \I   u^{\ast}
$$
and we have $E(u^{\ast})=E^{\ast}$. If $u^{\ast}$ is the only eigenfunction on the energy level $E^{\ast}$, then we have convergence of the full sequence $u^{n}$.
\end{theorem} 
\begin{proof} 
The proof is similar to the arguments presented in~\cite[Th.~4.9]{HenP18ppt}. First, Theorem~\ref{theorem-energy-reduction} guarantees the existence of the limit $E^{\ast}:=\lim_{n\rightarrow \infty} E(u^n)$. Hence, $u^n$ is uniformly bounded in $V$ and we can extract a subsequence (still denoted by $u^n$) that converges weakly in $H^1(\D)$ and strongly in $L^p(\D)$ (for $p<6$) to a limit $u^{\ast} \in V$ with $\| u^{\ast} \|=1$. Since $J_{\sigma}(u^{\ast})$ is a real-linear operator that depends continuously on the data and which induces the coercive bilinear form  
$\langle (J(u)+\sigma\I)\hspace{3pt} \cdot \hspace{2pt} , \hspace{2pt} \cdot \hspace{2pt} \rangle$, we have that
$$
J_{\sigma}(u^{\ast})^{-1} \I u^n \rightarrow  J_{\sigma}(u^{\ast})^{-1} \I u^{\ast} \quad \mbox{strongly in } H^1(\D).
$$
Together with the strong convergence $u^{n} \rightarrow u^{\ast}$ in $L^4(\D)$, we conclude that
\begin{align*}
J_{\sigma}(u^{n})^{-1} \I u^n \rightarrow  J_{\sigma}(u^{\ast})^{-1} \I u^{\ast} \quad \mbox{strongly in } H^1(\D).
\end{align*}
This shows that
\begin{align*}
(\gamma_n)^{-1} =
 ( J_{\sigma}(u^{n})^{-1} \I u^n , u^n) \overset{n \rightarrow \infty}{\longrightarrow} ( J_{\sigma}(u^{\ast})^{-1} \I u^{\ast} , u^{\ast})  =: (\gamma^{\ast})^{-1}.
\end{align*}
Furthermore, we have seen in the proof of Theorem \ref{theorem-energy-reduction} (respectively Conclusion \ref{conc-tildeu}) that the strong energy reduction implies that for $n\rightarrow 0$
$$
\| \tilde{u}^{n+1} - u^{n} \|_{H^1(\D)} \rightarrow 0
$$
and consequently we have with $\tilde{u}^{n+1} = (1-\tau_n)u^n + \tau_n \gamma^n J_{\sigma}(u^{n})^{-1} \I u^n$ and the boundedness of $\tau_n$ that
$$
u^n = \gamma^n J_{\sigma}(u^{n})^{-1} \I u^n - \tau_n^{-1}(\tilde{u}^{n+1} - u^{n}) \ \rightarrow\ \gamma^{\ast} J_{\sigma}(u^{\ast} )^{-1} \I u^{\ast}
$$
strongly in $H^1(\D)$. Since we already know that $u^n$ converges weakly in $H^1(\D)$ to $u^{\ast}$, we can now conclude that this is even a strong convergence and we have
$$
J_{\sigma}(u^{\ast} ) u^{\ast} = \gamma^{\ast} \I u^{\ast}.
$$
This shows that $u^{\ast}$ is an eigenfunction with eigenvalue $\gamma^{\ast}$. The strong $H^1$-convergence also implies convergence of the energies, i.e., $E^{\ast} =\lim_{n\rightarrow \infty} E(u^n) = E(u^{\ast})$.
\end{proof}

For all sufficiently small $\tau_n$, Theorem~\ref{theorem-energy-reduction} proves the energy reduction and Theorem~\ref{theorem-global-convergence} global convergence. However, since we do not know a priori what a sufficiently small value for~$\tau_n$ is, we can combine the damped $J$-method with a line search algorithm that optimizes~$\tau_n$ in each iteration step such that the energy reduction is (quasi) optimal. Theorems~\ref{theorem-energy-reduction} and~\ref{theorem-global-convergence} show that such an optimal $\tau_n$ exists and that it does not degenerate to zero. We stress that finding such a $\tau_n$ does not require any additional inversions, which makes the procedure very cheap, cf.~Appendix~\ref{app:Tau} for details.
\begin{conclusion}[$J$-method with optimal damping]
Consider a shift $\sigma$ such that the assumptions of Theorem \ref{theorem-energy-reduction} are fulfilled. Given $u^n\in V$ with $\|u^n\|=1$ the next iteration is obtained by selecting the optimal damping parameter with
$$
\tau_{n} := \underset{0<\tau \le 2}{\mbox{\rm arg min}} \hspace{5pt} E\left( \frac{(1- \tau)u^n + \tau\, \gamma_n\, J_\sigma(u^n)^{-1} \I u^n}{\| (1-\tau)u^n + \tau\, \gamma_n\, J_\sigma(u^n)^{-1} \I u^n \|} \right)
$$
and defining $u^{n+1}$ as in~\eqref{eqn:JmethodDamped}. 
The approximations are energy diminishing and converge (up to a subsequence) strongly in $V$ to an $L^2$-normalized eigenfunction of the GPEVP. 
\end{conclusion}
Finally, if there is no rotation, we can even achieve guaranteed global convergence to the ground state provided that the selected initial value is non-negative.
\begin{proposition}[global convergence to ground state]
Assume the setting of Theorem~\ref{theorem-global-convergence}.  Furthermore, we consider that there is no rotation, i.e., $\Omega=0$, and a non-negative starting value $u^0 \in V$, i.e., $u^0 \ge 0$. 
If $\tau_n \le 1$ and if the shift parameter~$\sigma>0$ is selected sufficiently large, then the iterates $u^n$ of the damped $J$-method converge strongly in $H^1(\D)$ to the unique (positive) ground state $u^{\ast} \ge 0$.
\end{proposition}
\begin{proof}
If $\Omega=0$, then the problem can be fully formulated over $\R$ and admits a unique positive ground state $u^{\ast} \in V$, cf.~\cite{CCM10}. The only other ground state is $- u^{\ast}$. Furthermore, all excited states (i.e., all other eigenfunctions) must necessarily change their sign on $\D$, cf.~\cite[Lem.~5.4]{HenP18ppt}. Hence, if we can verify that the iterates of the damped $J$-method~\eqref{eqn:JmethodDamped} preserve positivity, then Theorem \ref{theorem-global-convergence} guarantees that the global $H^1$-limit must be the desired ground state $u^{\ast}$. 

Recall the damped $J$-iteration from~\eqref{eqn:JmethodDamped}, which (over $\R$) reduces to
\begin{align*}
\langle J_{\sigma}(u^n) v , w  \rangle
= \langle S_{\sigma} v , w \rangle + \langle G v , w \rangle,
\end{align*}
where $S_{\sigma}$ is the linear self-adjoint operator given by
$$
\langle S_{\sigma} v , w \rangle := ( \nabla v , \nabla w )_{L^2(\mathcal{D})} + (  \hspace{2pt}(W+\sigma + 3 \kappa  |u^n|^2 ) \hspace{2pt}v, w )_{L^2(\mathcal{D})}
$$
and $G$ characterizes the non-symmetric rank-$1$ remainder, i.e.,
$$
\langle G v , w \rangle := - (u^n , v )_{L^2(\mathcal{D})} \hspace{2pt} (f ,  w )_{L^2(\mathcal{D})},\qquad
\mbox{where }  f := 2 \kappa\, |u^n|^2 u^n .
$$
Analogously to the Sherman--Morrison formula for matrices~\cite{SheM50}, we can see that 
\[
  (S_{\sigma} + G)^{-1}
  =
  S_{\sigma}^{-1} -\frac{S_{\sigma}^{-1} \circ G \circ S_{\sigma}^{-1}  }{1 -  (u^n , S_{\sigma}^{-1} \I f )_{L^2(\mathcal{D})} }.
\]
Consequently we can write the effect of the inverse as
$$
J_{\sigma}(u^n)^{-1} \I v =  (S_{\sigma} + G)^{-1} \I v 
= 
S_{\sigma}^{-1} (\I v) 
- \frac{S_{\sigma}^{-1} \circ G \circ S_{\sigma}^{-1} (\I v) }{1 -  (u^n , S_{\sigma}^{-1} \I f )_{L^2(\mathcal{D})}   }.
$$
Since $S_{\sigma}$ is a self-adjoint elliptic operator, it preserves positivity, i.e., we have $S_{\sigma}^{-1} \I v \ge 0$ if $v\ge 0$. This immediately follows by writing the action of $S_{\sigma}^{-1}$ as an energy minimization problem. Starting (inductively) from a function $u^n\ge0$ we conclude that
$$
S_{\sigma}^{-1} (\I u^n) \ge0 \qquad \mbox{and} \qquad
- S_{\sigma}^{-1} \circ G \circ S_{\sigma}^{-1} (\I u^n) \ge 0.
$$
If we can ensure that $1 -  (u^n , S_{\sigma}^{-1} \I  f )_{L^2(\mathcal{D})} >0$, then we have $J_{\sigma}(u^n)^{-1} \I u^n \ge 0$. Let us hence consider $(u^n , S_{\sigma}^{-1} \I  f )_{L^2(\mathcal{D})}$, for which we obtain
$$
| (u^n , S_{\sigma}^{-1} \I  f )_{L^2(\mathcal{D})} | \le \|u^n\| \, \|S_\sigma^{-1}\|_{\mathcal{L}(V^*,V)} \| \I f\|_{V^*}
 \le 2 \kappa\, \|u^n\|^3_{L^6(\D)} \, \|S_\sigma^{-1}\|_{\mathcal{L}(V^*,V)}.
$$
Since $S_{\sigma}$ is self-adjoint, we have
$$
\|S_\sigma^{-1}\|_{\mathcal{L}(V^*,V)} = 1/\lambda_\text{min}(S_\sigma) = 1/(\sigma + \lambda_\text{min}(S_0))\le 1/\sigma,
$$
where $ \lambda_\text{min}(S_0)>0$ is the minimal eigenvalue of $S_0$. Consequently, if the shift is such that
\[
 2 \kappa\, \|u^n\|^3_{L^6(\D)} < \sigma,
\]
then we have positivity of $J_{\sigma}(u^n)^{-1} \I u^n$. 
Note that by the energy reduction property, we can bound $\|u^n\|^3_{L^6(\D)}$ uniformly for all $n$ by a constant that only depends on the initial energy $E(u^0)$. Together with the obvious positivity of $(1-\tau_n)u^n$ for $\tau_n\le 1$, we conclude the existence of a sufficiently large shift $\sigma$ so that $u^{n+1}\ge 0$ for all $n\ge 0$ and hence global convergence to the ground state.
\end{proof}
%
%
\section{Numerical experiments}\label{sect:numerics}
This section concerns the numerical performance of the proposed $J$-method enhanced by shifting and/or damping as outlined in Sections~\ref{ss:J-shifted} and~\ref{ss:J-damped}. As a general model, we seek critical points of the Gross-Pitaevskii energy~\eqref{eq:energy} in a bounded domain $\D=(-L,L)^2$ for some parameter $L>0$. The particular choice of $L$ as well as the other physical parameters $\Omega$, $W$, and $\kappa$ will be specified separately in the various model problems below. 

For the spatial discretization we always use bilinear finite elements on a Cartesian mesh of width $h = 2^{-8}L$. We will not investigate discretization errors with respect to the underlying PDE. For approximation properties of discrete eigenfunctions we refer to the analytical results presented in \cite{CCH18,CCM10,CGZ10,HMP14b} and to \cite{CDM14,HSW19,XiX16b} for a posteriori estimators and adaptivity.
Our focus is the performance of the iterative eigenvalue solver promoted in this paper. As a measure of accuracy we will use the~$L^2(\D)$-norm of the residual~$\A(u^n)u^n-\lambda^n \I u^n$ given an approximate finite element eigenpair $(\lambda^n,u^n)$. We will stop the solver whenever the residual falls below the tolerance $\mbox{\rm\small TOL}=10^{-8}$.  

For a better assessment of the performance of the $J$-method, we compare it with the projected $a_z$-Sobolev gradient flow introduced in \cite{HenP18ppt}: given~$u^0 \in V$ with $\| u^0 \|=1$, define for $n = 1, 2, \ldots$, 
\begin{align}
\label{GFaz-eqn}\nonumber
\hat{u}^{n+1} &:= \mathcal{A}(u^n)^{-1}\mathcal{I} u^n,\qquad 
\gamma_n^{-1}:=\langle A(u^n,\hat{u}^{n+1}) , \hat{u}^{n+1} \rangle,
\\
u^{n+1} &:= \frac{(1 - \tau_n) u^n + \tau_n \gamma_n\hat{u}^{n+1}  }{\| (1 - \tau_n) u^n + \tau_n \gamma_n\hat{u}^{n+1} \|}.
\end{align}
We will refer to this approach as the $\A$-method (without shift). Note that every step of the~$\A$-method can be interpreted as an energy minimization problem such that the iteration is positivity preserving for the case without rotation. Thus guarantees global convergence to the ground state for every nonnegative starting value.
The corresponding shifted version, which we will also consider in the experiments, considers~$\mathcal{A}_\sigma(u^n) := \mathcal{A}(u^n) + \sigma \I$ in place of~$\mathcal{A}(u^n)$ in the definition of~$\hat{u}^{n+1}$ and a corresponding adjustment of the normalization factor~$\gamma_n$.
According to the numerical experiments of \cite{HenP18ppt}, this method is representative for the larger class of gradient flows in terms of accuracy-cost ratios. The cost per iteration step for both $\A$- and $J$-method are proportional and of the same order. Tentatively, the $\A$-method is cheaper by a fixed factor, since (due to the rank-$1$ matrix that appears in the $J$-version) an additional linear system per step has to be solved when the Sherman-Morrison formula is used, cf.~Appendix~\ref{app:Jmatrix}. To what extent the computational overhead of the $J$-method can be reduced by suitable preconditioned iterative solvers is beyond the scope of the paper. Notwithstanding the above, the experiments below clearly show that the $J$-method easily compensates its possible computational overhead per step by a notedly smaller iteration count.
%
\subsection{Ground state in a harmonic potential}\label{ss:numexp:harm}
In the first model problem, we consider a harmonic trapping potential with trapping frequencies $1/2$, i.e., 
\begin{equation}\label{eq:harm}
W(x)=\tfrac{1}{2}\, |x|^2.
\end{equation}
The angular momentum $\Omega$ is set to zero and the repulsion parameter to~$\kappa = 1000$. The size of the domain is chosen as $L=8$. This is larger than the Thomas-Fermi radius of the problem which can be estimated as $R^{\mbox{\scriptsize TF}}=\sqrt{2} \hspace{2pt}(\kappa/\pi)^{1/4}\approx5.97$, cf.~\cite{Bao14}. 
We are interested in computing the ground state, i.e., the global minimizer $u^\text{gs}$ of the energy \eqref{eq:energy}. Note that, up to sign, $u^\text{gs}$ is the unique eigenfunction that corresponds to the smallest eigenvalue $\lambda^\text{gs}$ which is well-separated from the remaining spectrum. As an initial value for all variants of eigenvalue solvers we use the bi-quadratic bubble
\begin{equation}
\label{eq:bubble}
u^0(x) = (1-x_1^2/L^2)(1-x_2^2/L^2),
\end{equation} 
interpolated in the finite element space and normalized in $L^2(\D)$. For this simple model problem, there are certainly more sophisticated initial guesses such as the Thomas-Fermi approximation. Our uneducated initial guess marks an additional challenge. As ground state energy we computed $E^\text{gs}:= E(u^\text{gs})\leq 6.019$. The corresponding eigenvalue approximation is $\lambda^\text{gs}\approx 17.93$. Clearly, the accuracy of these numbers is limited by the choice of the discretization parameter $h$. Mesh adaptivity as used in \cite{HSW19} or a higher-order method would certainly help to improve on these numbers.
\begin{figure}
%
%
\begin{tikzpicture}

\begin{axis}[%
width=4.4in,
height=2.1in,
at={(0.766in,1.146in)},
scale only axis,
xmin=0,
xmax=72,
xlabel=number of iterations,
xlabel style={below=-1.0mm},
ymode=log,
ymin=8e-10,
ymax=0.3,
yminorticks=true,
ylabel=$L^2$-norm of residual,
axis background/.style={fill=white},
xmajorgrids,
ymajorgrids,
yminorgrids,
legend style={legend cell align=left, align=left, draw=white!15!black, at={(0.98,0.90)}}
]
\addplot [color=darkBlue, line width=1.0pt, mark size=2.2pt, mark=square]
  table[row sep=crcr]{%
0	0.0158242956362844\\
1	0.00395644302988422\\
2	0.0011066667560253\\
3	0.000506839253224359\\
4	7.14071926563511e-06\\
5	5.56636724477214e-09\\
};
\addlegendentry{\small $J$-method}

\addplot [color=darkBlue, line width=1.0pt, mark size=2.2pt, mark=o]
  table[row sep=crcr]{%
0	0.0158242956362844\\
1	0.00395644302988422\\
2	0.00110666675602529\\
3	0.000506839253224356\\
4	0.000226752239879788\\
5	0.000115308897836691\\
6	6.50364651832702e-05\\
7	4.82025335432062e-05\\
8	2.42293648243155e-05\\
9	1.94180026216613e-05\\
10	9.18042668359323e-06\\
11	6.73799837931492e-06\\
12	3.2968691393139e-06\\
13	2.46492512906492e-06\\
14	1.22684550327289e-06\\
15	9.29532825622112e-07\\
16	4.67252116409565e-07\\
17	3.57816433520486e-07\\
18	1.55231285507349e-07\\
19	1.35403195264243e-07\\
20	5.89103504803052e-08\\
21	3.99167350490143e-08\\
22	2.05601803054216e-08\\
23	3.02298461391391e-08\\
24	5.46498107660896e-09\\
};
\addlegendentry{\small $J$-method ($\sigma = 0$)}

\addplot [color=darkBlue, line width=1.0pt,  mark size=1.2pt, mark=o]
  table[row sep=crcr]{%
0	0.0158242956362844\\
1	0.00852606469621754\\
2	0.00502103442564316\\
3	0.00317846147381111\\
4	0.00212094323880413\\
5	0.00146879972382851\\
6	0.00104396067920508\\
7	0.000755724994326532\\
8	4.37484510132343e-05\\
9	2.14554674522027e-07\\
10	6.07409409975354e-12\\
};
\addlegendentry{\small $J$-method ($\tau=1$)}

\addplot [color=darkRed, line width=1.0pt, mark size=2.2pt, mark=square]
  table[row sep=crcr]{%
0	0.0158242956362844\\
1	0.00749648848705238\\
2	0.00318688454902922\\
3	0.00242246170307509\\
4	0.00142941821512992\\
5	0.00114236251681556\\
6	0.000730253778012451\\
7	0.00706973780823292\\
8	0.015260057902494\\
9	0.405076630324779\\
10	0.0262520365239189\\
11	0.177272045219618\\
12	0.0398023149494973\\
13	0.0430609147672156\\
14	0.0305839586148964\\
15	0.115435407956597\\
16	0.0734019937634027\\
17	0.130495921251911\\
18	0.0875225028105712\\
19	0.109802428186084\\
20	0.0715144762593197\\
21	0.149378253312969\\
22	0.159409327423683\\
23	0.145382908613309\\
24	0.0554202651537804\\
25	0.113025780660821\\
26	0.0773853723793677\\
27	0.109217852553343\\
28	0.0793769787863513\\
29	0.157006886042607\\
30	0.0849841852822045\\
31	0.144770162535126\\
32	0.261383886642154\\
33	0.118366841668053\\
34	0.105517290535922\\
35	0.208304133712053\\
36	0.123722146104929\\
37	0.465891311860044\\
38	0.101019544675143\\
39	0.0903887543468415\\
40	0.133758375979905\\
41	0.22731626567503\\
42	0.148431112714033\\
43	0.156461054288301\\
44	0.149015573611169\\
45	0.103902894356913\\
46	0.393432656598703\\
47	0.157847235604441\\
48	0.225809930379258\\
49	0.162613105884789\\
50	0.123254815472761\\
51	0.230648531976592\\
52	0.203332928442542\\
53	0.320377353127517\\
54	0.142812095303523\\
55	0.138515472023808\\
56	0.160512177198367\\
57	0.189523702134398\\
58	0.126968445694019\\
59	0.149717030943557\\
60	0.142062847710212\\
61	0.236140290738525\\
62	0.0817618903618472\\
63	0.41977765192904\\
64	0.121190182334488\\
65	0.13922941061073\\
66	0.144678475061117\\
67	0.129157777282912\\
68	0.130211832449369\\
69	0.107464677925585\\
};
\addlegendentry{\small $\A$-method}

\addplot [color=darkRed, line width=1.0pt, mark size=2.2pt, mark=o]
  table[row sep=crcr]{%
0	0.0158242956362844\\
1	0.00749648848705238\\
2	0.00318688454902922\\
3	0.00242246170307509\\
4	0.00142941821512992\\
5	0.00114236251681556\\
6	0.000730253778012451\\
7	0.000588488772504272\\
8	0.000393715261008141\\
9	0.000313692270524883\\
10	0.000217507691014142\\
11	0.000172848687281282\\
12	0.000120675722631313\\
13	9.68286395123887e-05\\
14	6.77652814023575e-05\\
15	5.39853241486395e-05\\
16	3.83017753802565e-05\\
17	3.02845831840946e-05\\
18	2.17497987770317e-05\\
19	1.70643199866747e-05\\
20	1.23937184154063e-05\\
21	9.64592867962108e-06\\
22	7.01083710154535e-06\\
23	5.45517763665352e-06\\
24	4.00659801069053e-06\\
25	3.09154187848141e-06\\
26	2.27130545773379e-06\\
27	1.75239266430947e-06\\
28	1.28780277091755e-06\\
29	9.93472098645358e-07\\
30	7.37436775195689e-07\\
31	5.63944533741817e-07\\
32	4.18657546335854e-07\\
33	3.20156950458221e-07\\
34	2.37704501303785e-07\\
35	1.81774374771853e-07\\
36	1.34976887919872e-07\\
37	1.04194070242668e-07\\
38	7.58183009111934e-08\\
39	5.95363782895982e-08\\
40	4.28604432655446e-08\\
41	3.39357260606475e-08\\
42	2.41716415598201e-08\\
43	1.91124134212883e-08\\
44	1.39049490329403e-08\\
45	1.09246704601158e-08\\
46	7.35990718372007e-09\\
};
\addlegendentry{\small $\A$-method ($\sigma=0$)}

\addplot [color=darkRed, line width=1.0pt, mark size=1.2pt, mark=o]
  table[row sep=crcr]{%
0	0.0158242956362844\\
1	0.0072710995511535\\
2	0.00427190339627901\\
3	0.00279045594806374\\
4	0.00188979734066732\\
5	0.00137578073356561\\
6	0.000996066726065456\\
7	0.000764759467044929\\
8	0.000575263308396514\\
9	0.000453972810419112\\
10	0.000349969242883695\\
11	0.000280171937313812\\
12	0.000219645450327274\\
13	0.000177195233050558\\
14	0.000140576254277885\\
15	0.000113899587434053\\
16	9.11345496086099e-05\\
17	7.40339687686627e-05\\
18	5.96004224882412e-05\\
19	4.84991608967811e-05\\
20	3.92149516486922e-05\\
21	3.19476049508608e-05\\
22	2.5912151296414e-05\\
23	2.11271984512569e-05\\
24	1.71734150388456e-05\\
25	1.40102295612608e-05\\
26	1.14057805676487e-05\\
27	9.30876865819919e-06\\
28	7.58639613485127e-06\\
29	6.19341579693238e-06\\
30	5.0512110546214e-06\\
31	4.12459160543725e-06\\
32	3.36565379651307e-06\\
33	2.74864821628245e-06\\
34	2.24368665104724e-06\\
35	1.83255602826812e-06\\
36	1.49626078804455e-06\\
37	1.22217718201323e-06\\
38	9.98063788473235e-07\\
39	8.15281128042773e-07\\
40	6.65859936118479e-07\\
41	5.43935489266143e-07\\
42	4.44281838077633e-07\\
43	3.62939056893541e-07\\
44	2.964624494059e-07\\
45	2.42187721168833e-07\\
46	1.97836047109733e-07\\
47	1.61619141789764e-07\\
48	1.32025622743053e-07\\
49	1.07857143175844e-07\\
50	8.81095248692123e-08\\
51	7.19806467697792e-08\\
52	5.88024880159081e-08\\
53	4.80385639152605e-08\\
54	3.92440876183212e-08\\
55	3.20604384104963e-08\\
56	2.61912796419896e-08\\
57	2.13969805422958e-08\\
58	1.74800203626619e-08\\
59	1.42803229276547e-08\\
60	1.16661900639833e-08\\
61	9.53071118384058e-09\\
};
\addlegendentry{\small $\A$-method ($\sigma=0, \tau=1$)}

\end{axis}
\end{tikzpicture}%
\caption{Computation of the ground state for a harmonic potential, cf.~Section~\ref{ss:numexp:harm} for details. The figure shows the $L^2(\D)$-norms of the residuals (logarithmic scale) vs.~the iteration count for several methods indicated in the legend.}
\label{fig:convHarmPot}
\end{figure}
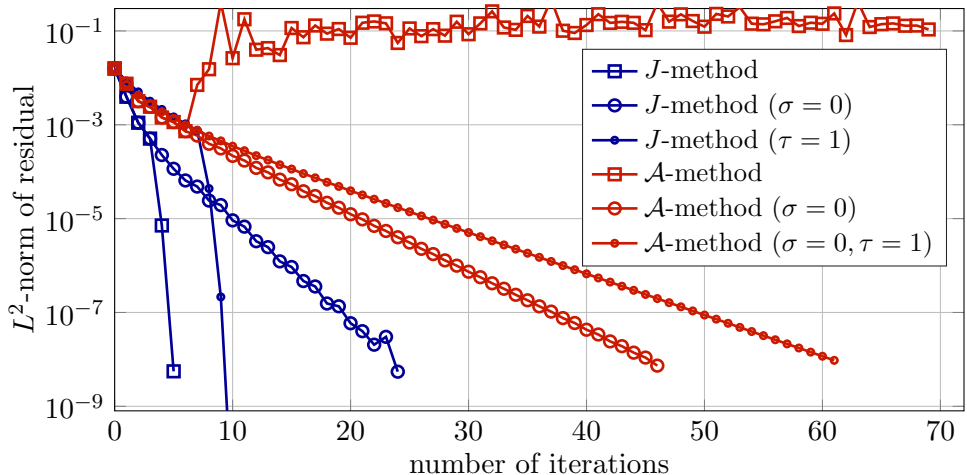

Figure~\ref{fig:convHarmPot} shows the evolution of the residuals during the iteration of several variants of the $J$ and $\A$-method. Unless specified differently, the general $J$- and $\A$-methods refer to a combination of damping and shifting. More precisely, we use damping for globalization of convergence until the residual falls below $10^{-3}$. Then we freeze the time step $\tau=1$ and switch to a Rayleigh shift strategy to possibly accelerate convergence, i.e., in each step we choose the shift $\sigma = -\lambda^n$ to be the current eigenvalue approximation. According to our numerical experience the coexistence of damping and shifting is hard to control. The transition from damping to shifting is clearly seen in the convergence plot of Figure~\ref{fig:convHarmPot}. We observe linear but fairly slow convergence in the damped phase. As soon as we switch to shifting, the convergence is beyond linear. The $\A$-method performs similarly in the damping phase but diverges as soon as the shift is turned on. It is this phenomenon already observed in \cite{JarKM14} in a less extreme characteristic (see also the second model problem below) that motivated the derivation of a shift-sensitive $J$-method.  
Note that the improved rate can be explained by its connection to Newton's method, cf.~[JKM14, Sect.~3.4]. Convergence proofs with higher rate, however, are only given in the discrete setting for linear and particular nonlinear eigenvalue problems~\cite{Osb64,MehV04}. 

In Figure~\ref{fig:convHarmPot} we also show results for the variants of the $J$- and $\A$-methods where either $\tau$, $\sigma$, or both are fixed. Their performance is in between the aforementioned combined approaches. 
%
\subsection{Exponentially localized ground state in a disorder potential}\label{ss:numexp:disorder}
In the second model problem the non-negative external potential $W$ reflects a high degree of disorder and the repulsion parameter~$\kappa$ is small. In this situation, the low-energy eigenstates essentially localize in the sense of an exponential decay of their moduli. 

The numerical approximation of localized Schr\"odinger eigenstates in the linear case, i.e., for $\kappa=0$, has recently caused a large interest in the fields of computational physics and scientific computing~\cite{FilM12,ArnDJMF16,Ste17,ArnDFJM19,XieZO19}. In particular, the results of \cite{AltHP20} provide a mathematical justification of the observed localization. In the nonlinear case the phenomenon is still observable but locality deteriorates with increasing interaction~\cite{AltPV18,AltP19,AHP20ppt}. Here, we choose $\kappa=1$ which leads to a fairly localized ground state as it can be seen in Figure~\ref{fig:DisPot} (right). Its computation turns out to be much more challenging than in the case of a harmonic trapping potential in the sense that convergence rates are slower and iteration counts larger. 
This is related to a possible clustering of the lowermost eigenvalues. 
In particular, we expect $\lambda_1/\lambda_2 \approx 1$ in this example such that shifting can provide a considerable speed up. 
Due to the small repulsion parameter~$\kappa$, however, we expect a significant gap within the first few eigenvalues as in the linear case.
\begin{figure}
\frame{\includegraphics[width=.4\textwidth]{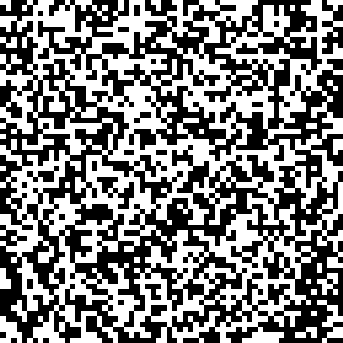}}\hspace{0.9em}
\frame{\includegraphics[width=.4\textwidth]{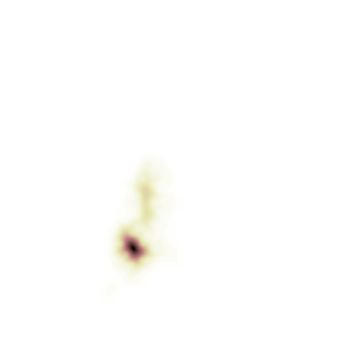}}
\caption{Exponentially localized ground state in a disorder potential, cf.~Section~\ref{ss:numexp:disorder}. Left: Disorder potential (random i.i.d.~checkerboard) taking values $0$ (white) and $(2\varepsilon L)^{-2}$ (black) for parameters $L=8$, $\varepsilon=2^{-6}$. Right: Corresponding ground state density for $\Omega = 0$, $\kappa = 1$.}
\label{fig:DisPot}
\end{figure}
\begin{figure}
\input{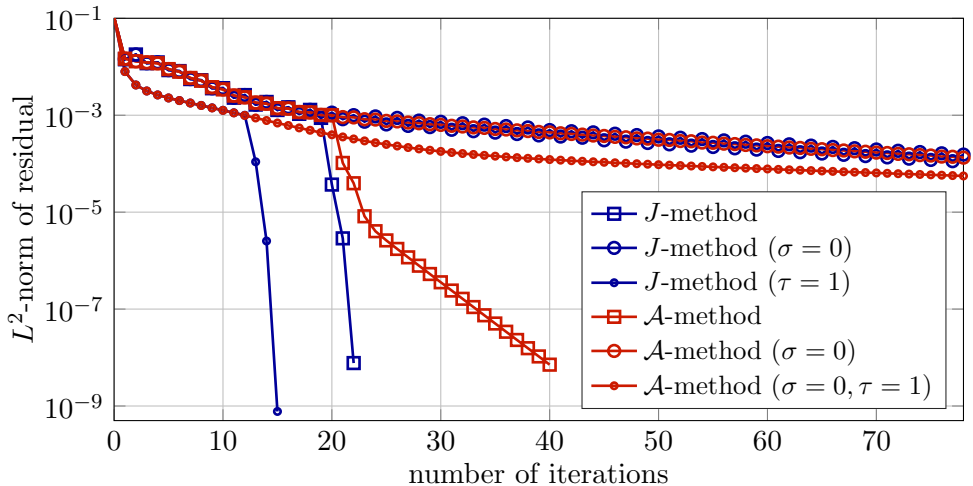}
\caption{Computation of the exponentially localized ground state in a disorder potential, cf.~Section~\ref{ss:numexp:disorder} for details. The figure shows the $L^2(\D)$-norms of the residuals (logarithmic scale) vs.~the iteration count for several variants of $J$- and $\A$-method.}
\label{fig:convDisPot}
\end{figure}

We have tested the same solvers as in the previous subsection. The $J$-method involving shifting performs best by far. Surprisingly, the variant without adaptive time step in the damping phase even performed better. This is no contradiction with the theory as we are showing residuals rather than energies. Moreover, a locally optimal energy decrease does not necessarily lead to better global performance. Still the difference is not too big. As a general recommendation from our numerical experience we would favor to use an adaptive time step because it was more robust. 
Another difference with regard to the harmonic potential is that, this time, the~$\A$-method reacts upon shifting in a positive way. For a few steps the convergence is indeed accelerated. However, thereafter the method turns back to a linear regime of convergence which cannot compete with the shifted $J$-method.  
Nevertheless, the $\A$-method does not fail completely as seen in the previous example. This may be caused by the small value of~$\kappa$, compared with the considered potential.
%
\subsection{Vortex lattices in a fast rotating trap}\label{ss:numexp:vortex}
We close the numerical illustration of the $J$-method with a qualitative study of vortex lattice states in the present of fast rotating potentials. We choose a harmonic potential~$W$ as in \eqref{eq:harm} and set $\Omega=0.99$, $\kappa= 1000$, and the size of the computational domain to $L=10$. 

We have computed four different eigenfunctions using the $J$-method. This was only possible using the shift-sensitivity of the $J$-method. We shall briefly describe the computational parameters. We use the bi-quadratic bubble \eqref{eq:bubble} as the initial value for all computations. To compute the (tentative) ground state $u_1$ (see Figure~\ref{fig:vortex}, upper left) we used the combined strategy as before. However, we switched from damping to shifting only once the residual falls below $10^{-6}$. Switching earlier led to states of higher energy. E.g., switching at a tolerance of $10^{-3}$ lead to the  eigenfunction $u_2$ depicted in the upper right of Figure~\ref{fig:vortex}. It is interesting to observe that while $E_1:=E(u_1)<E(u_2):=E_2$ the corresponding eigenvalues are ordered the other way around. 
\begin{figure}
\frame{\includegraphics[width=.38\textwidth]{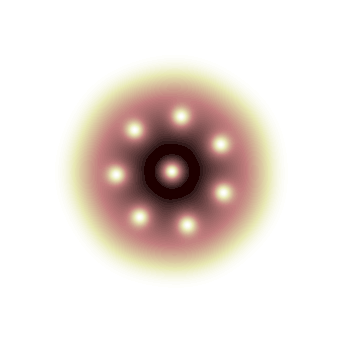}}\hspace{2ex}
\frame{\includegraphics[width=.38\textwidth]{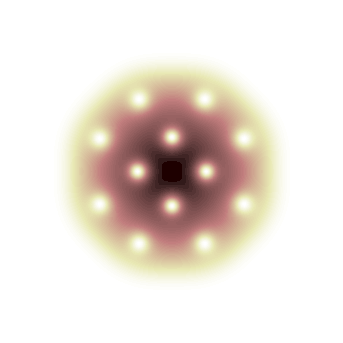}}\\[-4ex]
$\lambda_1 = 15.6094,\;E_1=5.3616$ \hspace{9ex} $\lambda_2=15.5470,\;E_2=5.3871$\\[3ex]
\frame{\includegraphics[width=.38\textwidth]{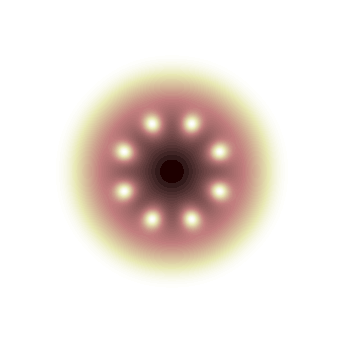}}\hspace{2ex}
\frame{\includegraphics[width=.38\textwidth]{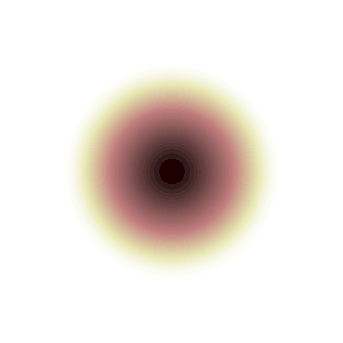}}\\[-4ex]
$\lambda_3=15.6434,\; E_3=5.3937$ \hspace{9ex} $\lambda_4= 17.9299,\;E_4=6.0188$ \\[2ex]
\caption{Computation of vortex lattices in a fast rotating trap at different energy levels, cf.~Section~\ref{ss:numexp:vortex}. The parameters are $L=10$, $\Omega = 0.85$, $\kappa = 1000$. The upper left figure depicts the density of the tentative ground state. The other three figures show densities corresponding to excited states. }
\label{fig:vortex}
\end{figure}

Two further excited states are found by limiting the adaptive shift to the interval $[15.0,15.6]$ for $u_3$ (see lower left of Figure~\ref{fig:vortex}) and to the interval $[15.2,15.45]$ for the state $u_4$ that does not show any vortices (see lower right of Figure~\ref{fig:vortex}). In both cases we used the lower end of the interval as the shift in the damping phase and we switched to adaptive shifting at residual tolerance $10^{-4}$ for $u_3$ and $10^{-3}$ for $u_4$. Note that $u_1$ seems to be the global energy minimizer of the (discretized) problem but the exited states $u_2,u_3,u_4$ do not necessarily represent the next higher energy levels $2$ to $4$ but some levels of higher energy.

From this rather complicated derivation one can see that it is by no means trivial to compute these excited states. We shall also say that it is not always easy to control the shifting. If one shifts too early in the sense that the approximation is not yet close to the target eigenfunction (e.g.~in terms of number of vortices) the procedure may fail completely. 
Despite this difficulty which is intrinsic to the nonlinear eigenvalue problem at hand, the $J$-method along with shifting and damping enables the selective approximation of excited states as well as the amplification of convergence beyond linear rates in the spirit of the Rayleigh quotient iteration even in this challenging regime of vortex pattern formation.  
%
\section{Conclusion}
In this paper we have generalized the $J$-method proposed in~\cite{JarKM14} to the abstract Hilbert space setting. This gives rise to a variational formulation that is straightforwardly accessible by Galerkin-type discretizations, e.g., based on finite elements. 
Moreover, we have transferred the proof of local convergence of the $J$-method from the discrete setting (cf.~\cite{JarKM14}) to the abstract setting and recovered a quantitative convergence rate that depends on the spectral shift. Since this fast convergence is indeed a local feature, we have proposed a damped $J$-method. For the GPEVP, the damping step can be seen as a discretization of a generalized gradient flow and guarantees reduction of the energy associated to the Gross-Pitaevskii operator. This energy reduction is the key to the global convergence of the damped method. 

We have proposed a combined strategy of damping and shifting, depending on the residual error. The damping part guides the iterates to a sufficiently small neighborhood of an eigenfunction. Therein, the shifting significantly improves the linear rate of convergence. With a Rayleigh-type shifting strategy remarkable speed-ups beyond linear convergence are observed. In numerical experiments we have demonstrated the excellent performance of the arising method and its suitability for both the computation of ground states and the selective computation of excited states. We believe that the proposed strategy can be also an efficient tool for treating other types of eigenvalue problems with nonlinearities, in particular those that can be rephrased as finding the critical points of constraint energy minimization problems. %
%
\newcommand{\etalchar}[1]{$^{#1}$}
\def\cprime{$'$}

%
%
\begin{appendix}
\section{Energy-diminishing step size control}\label{app:Tau}
We consider the damped $J$-method \eqref{eqn:JmethodDamped} in the case of the Gross-Pitaevskii equation. In order to implement an efficient step size control, we consider the function
$$f(\tau):=  E\left( \frac{ (1- \tau)u^n + \tau\, \gamma_n\, J_\sigma(u^n)^{-1} \I u^n } {\| (1-\tau)u^n + \tau\, \gamma_n\, J_\sigma(u^n)^{-1} \I u^n \| } \right)
$$
that we want to minimize for $\tau \in (0,2)$. Based on $u^n$, we compute 
$$
w^n := \gamma_n\, J_\sigma(u^n)^{-1} \I u^n.
$$
Note that this implies
\begin{align*}
\int_{\D} | (1 - \tau) u^n + \tau w^n |^4 \dx
= (1-&\tau)^4 \int_{\D} |u^n|^4 \dx
+ 4\, (1-\tau)^3 \tau \hspace{2pt} \int_{\D} \Re( u^n \overline{w^n}) \hspace{2pt} |u^n|^2 \dx \\
&+ (1-\tau)^2 \tau^2 \int_{\D} 2 |w^n|^2 |u^n|^2 
+ 4\, |\Re(u^n \overline{w^n} )|^2 \dx \\
&+ 4\, \tau^3 (1-\tau) \hspace{2pt} \int_{\D} \Re( u^n \overline{w^n})  \hspace{2pt} |w^n|^2 \dx  
+ \int_{\D} \tau^4 |w^n|^4 \dx.
\end{align*}
With this, we precompute various terms, which are given by
\begin{align*}
\alpha_0 &:= \int_{\D} |\nablaR u^n|^2 + \WR |u^n|^2\dx,\quad \alpha_1 := 2 \int_{\D} \Re\left( \nablaR u^n \cdot \overline{\nablaR w^n} + \WR \hspace{2pt} u^n \hspace{2pt} \overline{w^n} \right) \dx,\\
\alpha_2 &:= \int_{\D} |\nablaR w^n|^2 + \WR |w^n|^2\dx
\end{align*} 
as well as
\begin{align*}
\beta_0 &:= \frac{\kappa}{2} \int_{\D} |u^n|^4\dx, \quad \beta_1 := 2 \kappa \int_{\D} \Re( u^n \overline{w^n}) \hspace{2pt} |u^n|^2\dx , \quad \beta_2 := \kappa \int_{\D} |w^n|^2 |u^n|^2 + 2 |\Re(u^n \overline{w^n} )|^2 \dx,\\
\beta_3 &:= 2 \kappa \int_{\D}  \Re( u^n \overline{w^n})  \hspace{2pt} |w^n|^2 \dx, \quad \beta_4 :=  \frac{\kappa}{2} \int_{\D} |w^n |^4\dx
\end{align*}\and
\begin{align*}
\zeta_0 &:= \int_{\D} |u^n|^2\dx,  \qquad 
\zeta_1 := 2 \int_{\D} \Re( u^n \hspace{2pt} w^n) \dx,\qquad \zeta_2 := \int_{\D} |w^n|^2\dx.
\end{align*}
Note that the terms $\alpha_i$, $\beta_i$, and $\zeta_i$ have to be computed only once per time step (e.g.~within one loop over the grid elements). Finally, we define the function
$$
s^n(\tau ) 
:= \Biggl( \sum_{i,j\ge 0:\;i+j=2} (1-\tau)^i\, \tau_n^j\, \zeta_{j} \Biggr)^{-1/2}
$$
and observe that $f(\tau)$ is given by 
\begin{align*}
f(\tau) =
\frac{1}{2}\, \Biggl( \sum_{i,j\ge 0:\;i+j=2} |s^n(\tau)|^2 (1-\tau)^i\, \tau^j\, \alpha_{j}
+ \sum_{i,j\ge 0:\;i+j=4} |s^n(\tau)|^4 (1-\tau)^i\, \tau^j\, \beta_{j} \Biggr).
\end{align*}
After precomputing $\alpha_i$, $\beta_i$, and $\zeta_i$, the function $f(\tau)$ can be cheaply evaluated. The minimization step, i.e., $\tau_n := \mbox{arg min} \hspace{2pt} \{ f(\tau) | \hspace{3pt} \tau \in (0,2) \}$ can be easily implemented using, e.g., a golden section search. Note that the energy of $u^{n+1}$ is now given by $f(\tau_n)$.
	
\section{Matrix representation of the $\mathcal{J}$-operator}\label{app:Jmatrix}
The $\mathcal{J}$-operator in case of the GPEVP was derived in Section~\ref{sect:localGPE:Joperator}  
and we consider the iteration given by \eqref{eqn:JmethodDamped-new}. For the implementation we need to discuss the handling of the nonlinear terms.  Using the identity~$(u \overline{v} + 2 v \overline{u}) u\, \overline{w} = 2 \Re\left( u \overline{v} \right)u\, \overline{w}  + |u|^2 v\, \overline{w}$, we end up with the integrals  
\begin{align*}
I_1 := \int_{\D} \Re (u \overline{v})\, u\, \overline{w} \dx, \qquad
I_2 := \int_{\D} |u|^2 v\, \overline{w} \dx, \qquad 
I_3 := \Hscapro{u}{v} \int_{\D} |u|^2 u\, \overline{w} \dx.
\end{align*}
Note that we only need to consider real test functions $w$ and decompose $u, v$ into its real and imaginary part, i.e., $u=u_R + \ci u_I$, $v=v_R + \ci v_I$. This is also done in the finite element discretization, i.e., we work with real vectors of double dimension. For the first integral we note that 
\begin{align*}
\int_{\D} \Re(u \overline{v})\, u\, \overline{w} \dx 
&= \int_{\D} \big(u_R v_R + u_I v_I \big)\cdot \big(u_R w + \ci u_I w \big) \dx  \\
&= \int_{\D} \big(v_R u_R^2 + v_I u_Ru_I + \ci v_Ru_Ru_I + \ci v_Iu_I^2 \big) w \dx.  
\end{align*}
Introducing $M_{v}$ as the mass matrix weighted by $v$, this leads to the matrix representation 
\[
\begin{bmatrix} M_{u_Ru_R} & M_{u_Ru_I} \\ 
M_{u_Ru_I} & M_{u_Iu_I} \end{bmatrix}
\]

For the second term we have~$|u|^2 v\, \overline{w} = |u|^2 \big(v_R w + \ci v_I w \big)$. Thus, the corresponding finite element matrix is block diagonal and reads $\text{diag}(M_{|u|^2}, M_{|u|^2})$. 
Finally, 
we have 
\begin{align*}
I_3
= \Hscapro{u}{v} \int_{\D} |u|^2 u\, \overline{w}\dx 
= \big[(u_R,v_R)_{L^2(\D)} + (u_I,v_I)_{L^2(\D)} \big] \int_{\D} |u|^2 \big(u_R w + \ci u_I w \big) \dx.
\end{align*}
A simple rearrangement shows that this corresponds to the rank-one matrix 
\[
\begin{bmatrix} M u_R u_R^T M_{|u|^2} & M u_R u_I^T M_{|u|^2}  \\ 
M u_I u_R^T M_{|u|^2} & M u_I u_I^T M_{|u|^2} \end{bmatrix}
= \begin{bmatrix} M u_R \\ M u_I \end{bmatrix}	
\begin{bmatrix} M_{|u|^2} u_R \\ M_{|u|^2} u_I \end{bmatrix}^T.
\]

In total, a finite element discretization of the $J$-method calls for a solution of a linear system which is decomposed of several sparse matrices and the latter rank-$1$ update. This can be easily inverted using the Sherman-Morrison formula, cf.~\cite{SheM50, Woo50}. 
\end{appendix}
\end{document}